\documentclass{article}
\usepackage{amsmath}
\usepackage{amsfonts}
\usepackage{amsthm}
\usepackage{enumerate}
\usepackage{amssymb}

\newtheorem{thm}{Theorem}[section]
\newtheorem{prop}[thm]{Proposition}
\newtheorem{cor}[thm]{Corollary}
\newtheorem{lem}[thm]{Lemma}

\newtheorem*{thmnn}{Theorem}

\theoremstyle{remark}
\newtheorem{rmk}[thm]{Remark}

\theoremstyle{definition}
\newtheorem{defn}[thm]{Definition}

\begin{document}

\title{Jacobi Forms of Indefinite Lattice Index}

\author{Shaul Zemel}

\maketitle

\section*{Introduction}

Jacobi forms play a significant role in the theory of modular forms. One place where they naturally show up is as the Fourier--Jacobi coefficients of Siegel modular forms. They were systematically studied, in the classical case (of integer index) in \cite{[EZ]}, where in particular they were shown to have theta decompositions. This essentially means that unary theta functions are the most basic Jacobi forms, and all the others can be constructed from them using modular forms of weight $\frac{1}{2}$ less. The theta decomposition thus establishes an isomorphism between holomorphic Jacobi forms and holomorphic vector-valued modular forms, involving the Weil representation dual to that of a positive definite lattice of rank 1. The latter are also known to be related to the Kohnen plus-spaces of modular forms from \cite{[K1]}, \cite{[K2]}, and others (see also \cite{[LZ]} for generalizations of this connection). Indeed, some initial special cases of what is now known as the Borcherds lift were stated in terms of Jacobi forms, and this connection lies underneath the translation to this interpretation (after adding copies of the lattice $II_{1,1}$).

\smallskip

The notion of Jacobi forms has been generalized in several directions. First, Fourier--Jacobi coefficients of Siegel modular forms of higher rank are Jacobi forms in which both variables are now taken from higher-dimensional varieties: $\tau$ from a Siegel upper half-plane of a larger degree, and $\zeta$ from a space of complex matrices. The paper \cite{[Zi]} considers some of the properties of the classical Jacobi forms, and shows how they extend to this more general setting. Closely related objects are Jacobi forms (with $\tau$ in the usual upper half-plane $\mathcal{H}$ again) in which the index is no longer an integer, but rather a positive definite lattice (see, e.g., \cite{[BK]}). These are the subject of several recent papers, including the Hecke theory developed in one thesis \cite{[A]}, and later in another, \cite{[Mo]}.

Another type of Jacobi forms is related to the fact that sometimes the Kohnen plus-space is not related to the space of modular forms with the required Weil representation, but rather with its dual. The space of Jacobi forms that was seen to be isomorphic to modular forms with these representations is the space of \emph{skew-holomorphic} modular forms, see, e.g., \cite{[Sk]}. They are no longer holomorphic in $\tau$ (though they are in $\zeta$), and even their functional equations are not holomorphic, as they contain the absolute value, or equivalently the complex conjugate, of the factor of automorphy. They are, however, annihilated by an appropriate differential operator. Note that as modules over scalar-valued modular forms, they involve conjugation of the variable. The extension of this notion to the positive definite lattice case appears in, e.g., \cite{[H]}.

\smallskip

One may therefore ask whether there are Jacobi forms that are isomorphic to vector-valued modular forms with (dual) Weil representations that are associated with lattices that are not necessarily positive definite. Note that while it is true that every discriminant form can be obtained from a positive definite lattice, we consider the lattice as part of the data. We answer this question to the affirmative, once the lattice is supplemented with the choice of an element of its Grassmannian. The construction is based on a Jacobi version of the Siegel theta function defined in, e.g., Section 4 of \cite{[Bor]}, and indeed, once the correct theta function is defined, the proofs work like in the classical setting. Let $L_{\mathbb{R}}$ and $L_{\mathbb{C}}$ denote the real and complex spaces $L\otimes\mathbb{R}$ and $L\otimes\mathbb{C}$ respectively, let $\Theta_{L}$ be the theta function defined in Equation \eqref{JacTheta}, and let $\Delta_{v_{-}}^{h}$ stand for the holomorphic complexification, defined before Proposition \ref{difeq}, of the Laplacian associated the negative definite space $v_{-} \subseteq L_{\mathbb{R}}$ associated with the Grassmannian element $v$. The main result, which appears as Theorem \ref{main} and Proposition \ref{pshol}, is as follows.
\begin{thmnn}
Let $L$ be an even lattice of signature $(b_{+},b_{-})$, take $v$ in the Grassmannian of $L_{\mathbb{R}}$, and let $\Theta_{L}(\tau,\zeta;v)$ denote the value at $\tau\in\mathcal{H}$ and $\zeta \in L_{\mathbb{C}}$ of the Jacobi--Siegel theta function of $L$ and $v$. Then the map taking a modular form $F$ with representation $\rho_{L}^{*}$ to its pairing with $\Theta_{L}$ defines an isomorphism between Jacobi forms of weight $(k,l)$ and index $(L,v)$ and modular forms $\big(k-\frac{b_{+}}{2},l-\frac{b_{-}}{2}\big)$ and representation $\rho_{L}^{*}$. Moreover, $F$ is weakly holomorphic if and only if its associated Jacobi form is annihilated by the operator $4\pi i\partial_{\overline{\tau}}-\Delta_{v_{-}}^{h}$, where $\Delta_{v_{-}}^{h}$ operates on the variable $\zeta$.
\end{thmnn}
We remark that all of our Jacobi forms are holomorphic in $\zeta$, and the fact that $v$ is included in the index affects their functional equations (both the modular and the ``periodic'' ones). Our examples are therefore different from those from \cite{[CWR]} and \cite{[WR]}, that are also isomorphic to the same vector-valued modular forms (in particular, the theta functions from the latter reference have singularities, while our theta functions do not). Moreover, the operator $\Delta_{v_{-}}^{h}$ disappears when $L$ is positive definite, so that the differential equation is equivalent to holomorphicity in this case. The precise definition is given in Definition \ref{Jacdef}. Note that we allow linear exponential growth of our modular forms at the cusp, so that the appropriate analogue might be some version of \emph{weak} Jacobi forms, but altering the growth conditions would produce analogues of the usual Jacobi forms in the same manner. In addition, the dependence of our Jacobi forms on the Grassmannian variable $v$ may lead to interesting results.

Note that the skew-holomorphic Jacobi forms of \cite{[Sk]} and \cite{[H]} are not covered in this setting, in particular because the differential operator defining them involves $\partial_{\tau}$ rather than $\partial_{\overline{\tau}}$. However, the vector-valued version of the conjugation of the variable combines with our result to generalize these Jacobi forms as well (see Proposition \ref{skewhol}).

\smallskip

We then establish the behavior of several operations on vector-valued modular forms and Jacobi forms, and relate them via this connection: Direct sums, arrow operators to and from a sub-lattice of finite index, partial substitution of 0 (which is related to an operation called theta contraction, introduced in \cite{[Ma]} and generalized in \cite{[Ze3]}), and products of Jacobi forms. Note that the latter action, which was very natural when the indices were integers, becomes more complicated to define when the index is a lattice, and even more delicate when this lattice is indefinite. On the level of the vector-valued modular forms it produces interesting maps, which are related to a special case of theta contraction.

We remark that the relation involving one of the arrow operators cannot be phrased in terms of our theta functions and Jacobi forms alone. We thus make use of the ``generalized modularity'' of the theta functions from Theorem 4.1 of \cite{[Bor]} for establishing an analogue of theta functions with characteristics (see, e.g., Section 1.3 of \cite{[FZ]} for their definition in general), and an appropriate combination of those gives the answer for the missing arrow operator.

\smallskip

This paper is divided into 4 sections. Section \ref{STheta} defines the Jacobi--Siegel theta functions that we work with, and proves some of its properties. Section \ref{JacForms} contains the definition of our Jacobi forms, and establishes the main correspondence. Section \ref{Oper} investigates the behavior of the operations on both types of objects, and Section \ref{Char} considers the theta functions with characteristics.

\smallskip

I would like to thank B. Williams for an interesting discussion, mainly around Theorem \ref{MFprodJac}, as well as to M. Raum for referring me to \cite{[WR]} and \cite{[CWR]}. I am also very grateful to the two referees, for many valuable suggestions that improved the presentation of this paper.

\section{Jacobi--Siegel Theta Functions \label{STheta}}

A \emph{lattice} is a free $\mathbb{Z}$-module $L$ of finite rank, with a non-degenerate bilinear map taking $\lambda$ and $\mu$ in $L$ to their pairing $(\lambda,\mu)\in\mathbb{Z}$. The \emph{signature} of $L$ is the
signature $(b_{+},b_{-})$ of the corresponding real quadratic space $L_{\mathbb{R}}$. We say that $L$ is \emph{even} if the pairing $\lambda^{2}:=(\lambda,\lambda)$ is even for every $\lambda \in L$, or equivalently if the associated quadratic form $\lambda\mapsto\frac{\lambda^{2}}{2}$ is $\mathbb{Z}$-valued on $L$.

The pairing produces a canonical isomorphism \[L^{*}:=\operatorname{Hom}(L,\mathbb{Z})\cong\{\lambda \in L_{\mathbb{R}}|(\lambda,L)\subseteq\mathbb{Z}\} \subseteq L_{\mathbb{R}},\] so that the \emph{dual lattice} $L^{*}$ contains $L$. The quotient $D_{L}:=L^{*}/L$ is a finite group, called the \emph{discriminant group} of $L$. It carries the non-degenerate bilinear and quadratic forms with values in $\mathbb{Q}/\mathbb{Z}$ arising from those of $L$, and we denote them by the same symbols for those on $L$.

\smallskip

Theta functions of positive definite lattices are functions of the variable $\tau$ in the \emph{upper half-plane} $\mathcal{H}:=\{\tau=x+iy\in\mathbb{C}|y>0\}$. If $L$ is indefinite (i.e., when $b_{+}b_{-}>0$), then the theta function depends on another parameter, coming from the \emph{Grassmannian} of $L_{\mathbb{R}}$. This is the real (connected) manifold of dimension $b_{+}b_{-}$ that is defined by
\begin{equation}
\operatorname{Gr}(L_{\mathbb{R}}):=\big\{L_{\mathbb{R}}=v_{+} \oplus v_{-}\big|v_{+}\gg0,\ v_{-}\ll0,\ v_{+} \perp v_{-}\big\}, \label{Grassdef}
\end{equation}
namely the set of decompositions of $L_{\mathbb{R}}$ as the orthogonal direct sum of a positive definite subspace with a negative definite one. It is clear that for every element $v\in\operatorname{Gr}(L_{\mathbb{R}})$ we have $\dim v_{\pm}=b_{\pm}$, and that each $v_{\pm}$ determines the other vector space as $v_{\mp}=v_{\pm}^{\perp}$. It is not hard to verify that the Lie group $\operatorname{O}(L_{\mathbb{R}})$, as well as its connected component $\operatorname{SO}^{+}(L_{\mathbb{R}})$, act transitively on the Grassmannian $\operatorname{Gr}(L_{\mathbb{R}})$ from Equation \eqref{Grassdef} with maximal compact stabilizers. The latter is thus the \emph{symmetric space} of these groups. The former group contains the group $\operatorname{O}(L)$ of automorphisms of $L$ as a discrete subgroup, and we denote the intersection $\operatorname{O}(L)\cap\operatorname{SO}^{+}(L_{\mathbb{R}})$ by $\operatorname{SO}^{+}(L)$. A natural discrete group associated with $L$ is its \emph{discriminant kernel}, or \emph{stable orthogonal group}, which is defined by
\begin{equation}
\Gamma_{L}:=\ker\big(\operatorname{SO}^{+}(L)\to\operatorname{O}(D_{L})\big)=\big\{\mathcal{A}\in\operatorname{SO}^{+}(L)\big|\mathcal{A}\lambda-\lambda \in L\ \forall\lambda \in L^{*}\big\}. \label{GammaL}
\end{equation}
Finally, given a vector $\lambda \in L_{\mathbb{R}}$, we shall denote its orthogonal projection onto the spaces $v_{\pm}$ associated with an element $v\in\operatorname{Gr}(L_{\mathbb{R}})$ by $\lambda_{v_{\pm}}$.

\smallskip

The upper half-plane $\mathcal{H}$ carries a natural action of the group $\operatorname{SL}_{2}(\mathbb{R})$, in which the element $A=\binom{a\ \ b}{c\ \ d}$ operates as $\tau \mapsto A\tau:=\frac{a\tau+b}{c\tau+d}$. The denominator $j(A,\tau):=c\tau+d$ is called the \emph{factor of automorphy} for this action, and it satisfies the \emph{cocycle condition}
\begin{equation}
j(AB,\tau)=j(A,B\tau)j(B,\tau). \label{cocycle}
\end{equation}
This allows one to have an explicit description of the unique non-trivial double cover of that group: It is denoted by
\begin{equation}
\operatorname{Mp}_{2}(\mathbb{R}):=\big\{(A,\phi)\big|A\in\operatorname{SL}_{2}(\mathbb{R}),\ \phi:\mathcal{H}\to\mathbb{C}\text{ holomorphic},\ \phi^{2}(\tau)=j(A,\tau)\big\}, \label{Mp2R}
\end{equation}
and called the \emph{(real) metaplectic group}, in which the product rule is \[(A,\phi)\cdot(B,\psi):=\big(AB,(\phi \circ B)\cdot\psi\big).\] It is used for the natural definitions in the theory of modular forms of half-integral weight. The uniqueness of $\operatorname{Mp}_{2}(\mathbb{R})$ follows from the fact that topologically $\operatorname{SL}_{2}(\mathbb{R})$ is the product of $\mathcal{H}$ with the stabilizer $\operatorname{SO}(2)$, implying that its fundamental group is infinite cyclic and thus has a unique subgroup of index 2.

The \emph{integral metaplectic group} is the inverse image $\operatorname{Mp}_{2}(\mathbb{Z})$ of $\operatorname{SL}_{2}(\mathbb{Z})$ inside $\operatorname{Mp}_{2}(\mathbb{R})$. It contains the three special elements \[T:=\bigg(\binom{1\ \ 1}{0\ \ 1},1\bigg),\qquad S:=\bigg(\binom{0\ \ -1}{1\ \ \ 0\ },\sqrt{\tau}\in\mathcal{H}\bigg)\qquad\text{and}\qquad Z:=(-I,i),\] which generate it modulo the relations $S^{2}=(ST)^{3}=Z$ and the triviality of $Z^{4}$. To the discriminant group $D_{L}$ (or directly to the lattice $L$) one associates a \emph{Weil representation} $\rho_{L}$ of $\operatorname{Mp}_{2}(\mathbb{Z})$ on the space underlying the group ring $\mathbb{C}[D_{L}]$, which becomes unitary if we endow that space with the inner product in which the natural basis, which we denote by $\{\mathfrak{e}_{\gamma}\}_{\gamma \in D_{L}}$, is orthonormal. The explicit formulae for the action of the generators are
\begin{equation}
\rho_{L}(T)\mathfrak{e}_{\gamma}=\mathbf{e}\big(\tfrac{\gamma^{2}}{2}\big)\mathfrak{e}_{\gamma},\qquad\rho_{L}(S)\mathfrak{e}_{\gamma}=\frac{\mathbf{e}\big(\frac{b_{-}-b_{+}}{8}\big)}{\sqrt{|D_{L}|}}\sum_{\delta \in D_{L}}\mathbf{e}\big(-(\gamma,\delta)\big)\mathfrak{e}_{\delta}, \label{Weildef}
\end{equation}
and $\rho_{L}(Z)\mathfrak{e}_{\gamma}=i^{b_{-}-b_{+}}\mathfrak{e}_{-\gamma}$, where we write $\mathbf{e}(w)$ for $e^{2\pi iw}$ for every $w$ in $\mathbb{C}$ or in $\mathbb{C}/\mathbb{Z}$. For the explicit expression for $\rho_{L}(A,\phi)$ for every $(A,\phi)\in\operatorname{Mp}_{2}(\mathbb{Z})$, consult \cite{[Sch]}, \cite{[Str]}, or \cite{[Ze1]}.

\smallskip

The \emph{Siegel theta function} associated with $L$ is the function taking the variables $\tau\in\mathcal{H}$ and $v\in\operatorname{Gr}(L_{\mathbb{R}})$ to
\begin{equation}
\Theta_{L}(\tau;v):=\sum_{\lambda \in L^{*}}\mathbf{e}\bigg(\tau\frac{\lambda_{v_{+}}^{2}}{2}+\overline{\tau}\frac{\lambda_{v_{-}}^{2}}{2}\bigg)\mathfrak{e}_{\lambda+L}=\sum_{\gamma \in D_{L}}\Bigg[\sum_{\lambda \in L+\gamma}\mathbf{e}\bigg(\tau\frac{\lambda_{v_{+}}^{2}}{2}+\overline{\tau}\frac{\lambda_{v_{-}}^{2}}{2}\bigg)\Bigg]\mathfrak{e}_{\gamma}, \label{Thetadef}
\end{equation}
where the scalar-valued function multiplying $\mathfrak{e}_{\gamma}$ in Equation \eqref{Thetadef} is denoted by $\theta_{L+\gamma}(\tau;v)$. For the well-known convergence of the series from Equation \eqref{Thetadef}, as well as that from Equation \eqref{JacTheta} below, recall that the \emph{majorant} associated with $v$ is the positive definite bilinear form on $L_{\mathbb{R}}$ with respect to which the pairing of $\lambda \in L_{\mathbb{R}}$ with itself is $\lambda_{v_{+}}^{2}-\lambda_{v_{-}}^{2}$. Then the absolute value of the term associated with $\lambda$ has absolute value $\exp\big(-\pi y(\lambda_{v_{+}}^{2}-\lambda_{v_{-}}^{2})\big)$, and the discreteness of costes of $L$ inside $L_{\mathbb{R}}$ makes the convergence clear. Theorem 4.1 of \cite{[Bor]} states that the function $\tau\mapsto\Theta_{L}(\tau;v)$, with $v$ fixed, is a (typically non-holomorphic) modular form of weight $\big(\frac{b_{+}}{2},\frac{b_{-}}{2}\big)$ and representation $\rho_{L}$ (see Equation \eqref{modeq} below for the precise definition). It is also clear that for fixed $\tau$, the value of $\Theta_{L}(\tau;v)$ remains invariant under the action of the group $\Gamma_{L}$ from Equation \eqref{GammaL} on the variable $v$.

Jacobi forms involve another variable $\zeta$, where in case the index is the lattice $L$ it is taken from $L_{\mathbb{C}}$. We recall that our extension of the pairing on $L_{\mathbb{C}}$ is bilinear, and the quadratic form extends as well (neither involve complex conjugations). Adapting the definitions from the positive definite lattice index case from, e.g., Definition 3.32 of \cite{[Boy]} or Definition 2.3.2 of \cite{[A]} (among others), we define the \emph{Jacobi--Siegel theta function} that is associated with $L$ by
\begin{equation}
\Theta_{L}(\tau,\zeta;v):=\sum_{\gamma \in D_{L}}\Bigg[\sum_{\lambda \in L+\gamma}\mathbf{e}\bigg(\tau\frac{\lambda_{v_{+}}^{2}}{2}+\overline{\tau}\frac{\lambda_{v_{-}}^{2}}{2}+(\lambda,\zeta)\bigg)\Bigg]\mathfrak{e}_{\gamma}, \label{JacTheta}
\end{equation}
where the coefficient in front of $\mathfrak{e}_{\gamma}$ in Equation \eqref{JacTheta} is denoted by $\theta_{L+\gamma}(\tau,\zeta;v)$. It satisfies the following ``periodicity'' property in $\zeta$, generalizing the one given in Equation (2.32) of \cite{[A]}.
\begin{prop}
Consider $\tau\in\mathcal{H}$ and $v\in\operatorname{Gr}(L_{\mathbb{R}})$ as fixed, and take two elements $\sigma$ and $\nu$ from $L$. Then the Jacobi--Siegel theta function from Equation \eqref{JacTheta} satisfies, as a function of $\zeta \in L_{\mathbb{C}}$, the equality \[\Theta_{L}(\tau,\zeta+\tau\sigma_{v_{+}}+\overline{\tau}\sigma_{v_{-}}+\nu;v)=\mathbf{e}\bigg(-\tau\frac{\sigma_{v_{+}}^{2}}{2}-\overline{\tau}\frac{\sigma_{v_{-}}^{2}}{2}-(\sigma,\zeta)\bigg)\Theta_{L}(\tau,\zeta;v).\] \label{perTheta}
\end{prop}

\begin{proof}
Adding $\tau\sigma_{v_{+}}+\overline{\tau}\sigma_{v_{-}}+\nu$ to $\zeta$ in the summand associated with $\lambda$ in Equation \eqref{JacTheta} multiplies it by $\mathbf{e}\big(\tau(\lambda_{v_{+}},\sigma_{v_{+}})+\overline{\tau}(\lambda_{v_{-}},\sigma_{v_{-}})\big)$ (because the projections are orthogonal) as well as by $\mathbf{e}\big((\lambda,\nu)\big)=1$ (since $\nu \in L$ and $\lambda \in L^{*}$). Then the multipliers of $\tau$ and of $\overline{\tau}$ in the exponent are $\frac{\lambda_{v_{\pm}}^{2}}{2}+(\lambda_{v_{\pm}},\sigma_{v_{\pm}})=\frac{(\lambda+\sigma)_{v_{\pm}}^{2}}{2}-\frac{\sigma_{v_{\pm}}^{2}}{2}$, and we write $(\lambda,\zeta)$ as $(\lambda+\sigma,\zeta)-(\sigma,\zeta)$. This gives the desired multiplier, and since adding $\sigma \in L$ to the summation index $\lambda$ leaves the set $L+\gamma$ invariant for every $\gamma \in D_{L}$, we indeed obtain the original function from Equation \eqref{JacTheta}. This proves the proposition.
\end{proof}

We fix the Haar measure on $L_{\mathbb{R}}$ to be product of the usual Lebesgue measures in the coordinates with respect to an orthogonal basis with norms $\pm1$. Then, given a Schwartz function $f$ on $L_{\mathbb{R}}$, our normalization of its Fourier transform is defined to be $\hat{f}(\mu)=\int_{L_{\mathbb{R}}}f(\lambda)\mathbf{e}\big((\lambda,\mu)\big)d\lambda$. Then the \emph{Poisson Summation Formula}, as cited, for example, in the proof of Theorem 4.1 of \cite{[Bor]}, states that $\sqrt{|D_{L}|}\sum_{\lambda \in L}f(\lambda)=\sum_{\mu \in L^{*}}\hat{f}(\mu)$. Moreover, if $f(\lambda)=g(\lambda+\xi)$ for some element $\xi \in L_{\mathbb{R}}$, then we have $\hat{f}(\mu)=\hat{g}(\mu)\mathbf{e}\big(-(\mu,\xi)\big)$ as in, e.g., part 1 of Lemma 3.1 of \cite{[Bor]}, while part 2 of that lemma implies that if $g(\lambda)=h(\lambda)\mathbf{e}\big((\lambda,\eta)\big)$ for another element $\eta \in L_{\mathbb{R}}$ then $\hat{g}(\mu)=\hat{h}(\mu+\eta)$ (note that we use the opposite normalization to that of \cite{[Bor]}). This produces the more general form
\begin{equation}
\sum_{\lambda \in L}h(\lambda+\xi)\mathbf{e}\big((\lambda+\xi,\eta)\big)=\frac{1}{\sqrt{|D_{L}|}}\sum_{\mu \in L^{*}}\mathbf{e}\big(-(\mu,\xi)\big)\hat{h}(\mu+\eta), \label{PSF}
\end{equation}
where both sides are clearly determined by the image of $\xi$ in $L_{\mathbb{R}}/L$. We shall apply Equation \eqref{PSF} only to the function $h(\lambda)=\mathbf{e}\Big(-\frac{\lambda_{v_{+}}^{2}}{2\tau}-\frac{\lambda_{v_{-}}^{2}}{2\overline{\tau}}\Big)$ for fixed $\tau\in\mathcal{H}$ and $v\in\operatorname{Gr}(L_{\mathbb{R}})$, with varying values of $\xi$, $\eta$, and $\gamma$. Corollary 3.5 of \cite{[Bor]} shows (among earlier, simpler results) that its Fourier transform is given by $\hat{h}(\nu)=\mathbf{e}\big(\frac{b_{-}-b_{+}}{8}\big)\tau^{b_{+}/2}\overline{\tau}^{b_{-}/2}\mathbf{e}\Big(\tau\frac{\nu_{v_{+}}^{2}}{2}+\overline{\tau}\frac{\nu_{v_{-}}^{2}}{2}\Big)$.

Following the proof of Theorem 4.1 of \cite{[Bor]} or Theorem 3.1 on page 82 of \cite{[Boy]}, we also establish the following result, in which we recall that the derivative $j_{A}'$ of the function $j_{A}(\tau):=j(A,\tau)$ is a (real) constant, which is the lower left entry of the matrix $A$.
\begin{thm}
For every $(A,\phi)\in\operatorname{Mp}_{2}(\mathbb{Z})$ we have the equality \[\Theta_{L}\!\Big(A\tau,\tfrac{\zeta_{v_{+}}}{j(A,\tau)}+\tfrac{\zeta_{v_{-}}}{\overline{j(A,\tau)}};v\Big)\!=\! \phi(\tau)^{b_{+}}\overline{\phi(\tau)}^{b_{-}}\!\mathbf{e}\Big(\tfrac{j_{A}'\zeta_{v_{+}}^{2}}{2j(A,\tau)}+\tfrac{j_{A}'\zeta_{v_{-}}^{2}}{2\overline{j(A,\tau)}}\Big)\rho_{L}(A,\phi)\Theta_{L}(\tau,\zeta;v).\] \label{modTheta}
\end{thm}
Note the complex conjugation in the denominators under $\zeta_{v_{-}}$ and $\zeta_{v_{-}}^{2}$.
\begin{proof}
The usual argument, using Equation \eqref{cocycle} and its logarithmic derivative, shows that if the asserted equality holds for two metaplectic elements and every value of $\tau$ and $\zeta$ then it holds for their product, again for every $\tau$ and $\zeta$. It therefore suffices to verify the desired equality for $T$, $T^{-1}$, and $S$ (the latter being of finite order). For $T^{\pm1}$ we only add $\pm1$ to $\tau$ on the left hand side, with $\zeta$ fixed. But this only multiplies the summand associated with $\lambda$ in Equation \eqref{JacTheta} by $\mathbf{e}\big(\pm\frac{\lambda^{2}}{2}\big)$, and this is the same as $\mathbf{e}\big(\pm\frac{\gamma^{2}}{2}\big)$ in case $\lambda \in L+\gamma$ with $\gamma \in D_{L}$. We thus obtain just $\rho_{L}(T^{\pm})\Theta_{L}(\tau,\zeta;v)$, which is the desired result since $\phi=1$ and $j_{T^{\pm}}'=0$. Considering now the left hand side with $A=S$, the multiplier $\theta_{L+\gamma}$ in front of $\mathfrak{e}_{\gamma}$ in the corresponding Equation \eqref{JacTheta} is \[\sum_{\lambda \in L+\gamma}\mathbf{e}\bigg(-\frac{\lambda_{v_{+}}^{2}}{2\tau}+\frac{(\lambda_{v_{+}},\zeta_{v_{+}})}{\tau}-\frac{\lambda_{v_{-}}^{2}}{2\overline{\tau}}+\frac{(\lambda_{v_{-}},\zeta_{v_{-}})}{\overline{\tau}}\bigg),\] which we can write as
\begin{equation}
\mathbf{e}\bigg(\frac{\zeta_{v_{+}}^{2}}{2\tau}+\frac{\zeta_{v_{-}}^{2}}{2\overline{\tau}}\bigg)\sum_{\lambda \in L+\gamma}\mathbf{e}\bigg(-\frac{(\lambda-\zeta)_{v_{+}}^{2}}{2\tau}-\frac{(\lambda-\zeta)_{v_{-}}^{2}}{2\overline{\tau}}\bigg). \label{quadcomp}
\end{equation}
When $\zeta \in L_{\mathbb{R}}$ we apply Equation \eqref{PSF} with $\xi=\gamma-\zeta$ and $\eta=0$, which transforms the expression from Equation \eqref{quadcomp} into $\mathbf{e}\big(\frac{b_{-}-b_{+}}{8}\big)\tau^{\frac{b_{+}}{2}}\overline{\tau}^{\frac{b_{-}}{2}}\mathbf{e}\big(\frac{\zeta_{v_{+}}^{2}}{2\tau}+\frac{\zeta_{v_{-}}^{2}}{2\overline{\tau}}\big)\big/\sqrt{D_{L}}$ times $\sum_{\mu \in L^{*}}\mathbf{e}\Big(\tau\frac{\mu_{v_{+}}^{2}}{2}+\overline{\tau}\frac{\mu_{v_{-}}^{2}}{2}+(\mu,\zeta)-(\gamma,\mu)\Big)$. Recalling that $j_{S}'=1$, $j(S,\tau)=\tau$, and $\phi(\tau)=\sqrt{\tau}$, this gives the multiplier on the required right hand side as well as the external coefficient in the definition of $\rho_{L}(S)$ in Equation \eqref{Weildef}. By separating the sum over $L^{*}$ into the sums over the different cosets of $L$, it produces $\sum_{\delta \in D_{L}}\mathbf{e}\big(-(\gamma,\delta)\big)\theta_{L+\delta}(\tau,\zeta;v)$, as desired. Since both sides are, for fixed $\tau$, holomorphic functions of $\zeta \in L_{\mathbb{C}}$ that coincide on $L_{\mathbb{R}}$, they are the same function. This proves the theorem.
\end{proof}

\medskip

One of the most useful properties of holomorphic Jacobi theta functions is that they are solutions of the (holomorphic) heat equation. For non-holomorphic Jacobi forms there is a variety of differential operators, which act as weight changing operators---see, e.g., \cite{[CWR]}, \cite{[WR]}, \cite{[BRR]}, and \cite{[RR]}. We shall shorthand the partial derivative $\frac{\partial}{\partial z}$ to $\partial_{z}$ for every variable $z$, and given a real quadratic space $V$ of some dimension $d$, we denote by $\Delta_{V}$ the Laplacian operator on functions on $V$ (this means that if $\xi_{l}$, $1 \leq l \leq d$ a set of variables associated with a basis for $V$, and $\xi_{l}^{*}$, $1 \leq l \leq d$ correspond to the basis that is dual to the previous one by the pairing, then $\Delta_{V}=\sum_{l=1}^{d}\partial_{\xi_{l}}\partial_{\xi_{l}^{*}}$, and this is independent of the choice of basis). We shall write $\Delta_{V}^{h}$ for the holomorphic Laplacian operator on functions on $V_{\mathbb{C}}$, that is obtained from $\Delta_{V}$ by replacing each real derivative by its complex holomorphic counterpart. We remark that there are analogous operators $\Delta_{V}^{\mathbb{R}}$ and $\Delta_{V}^{\overline{h}}$, in which one or both of the real derivatives are replaced by an anti-holomorphic one (as in, e.g., \cite{[Ze2]}), but since all the functions considered in this paper will be holomorphic in $\zeta$, only the Laplacian $\Delta_{V}^{h}$ will be of interest.

We can now determine some useful differential equations that are satisfied by our Jacobi--Siegel theta functions.
\begin{prop}
The theta function from Equation \eqref{JacTheta} is annihilated by the differential operators $4\pi i\partial_{\tau}-\Delta_{v_{+}}^{h}$, $4\pi i\partial_{\overline{\tau}}-\Delta_{v_{-}}^{h}$, and $\partial_{\overline{\zeta}_{l}}$ for $1 \leq l \leq b_{+}+b_{-}$. \label{difeq}
\end{prop}

\begin{proof}
We shall consider the action of each operator on the summand associated with a fixed element $\lambda \in L^{*}$ in Equation \eqref{JacTheta}. On one hand, the derivatives $\partial_{\tau}$ and $\partial_{\overline{\tau}}$ multiply this summand by $\pi i\lambda_{v_{\pm}}^{2}$. On the other hand, consider an orthonormal basis for $v_{\pm}$, in which the coordinates of $\lambda_{\pm}$ are $\lambda_{l}$ with $1 \leq l \leq b_{\pm}$, and the coordinates of the projection of $\zeta$ onto the complexification of $v_{\pm}$ (in the natural complex bilinear form) are $\zeta_{l}$, $1 \leq l \leq b_{\pm}$. Then $\Delta_{v_{\pm}}^{h}=\pm\sum_{l=1}^{b_{\pm}}\partial_{\zeta_{l}}^{2}$, the action of $\partial_{\zeta_{l}}$ on the summand in question multiplies it by $\pm2\pi i\lambda_{l}$, and thus $\Delta_{v_{\pm}}^{h}$ multiplies it by $(2\pi i)^{2}$ times $\pm\sum_{l=1}^{b_{\pm}}\lambda_{l}^{2}=\lambda_{\pm}^{2}$. The vanishing under the first two operators now follows from a simple subtraction, and the vanishing under the last operators is just the holomorphicity of $\Theta_{L}$ as a function of $\zeta$. This proves the proposition.
\end{proof}
Recall that the Laplacian $\Delta_{L_{\mathbb{R}}}$ is the sum $\Delta_{v_{+}}+\Delta_{v_{-}}$ for every $v\in\operatorname{Gr}(L_{\mathbb{R}})$, and that the difference $\Delta_{v_{+}}-\Delta_{v_{-}}$ is the operator, that we denote by $\Delta_{L_{\mathbb{R}},v}$ following \cite{[Ze3]} and others, that corresponds to the space $L_{\mathbb{R}}$ endowed with the quadratic structure coming from the majorant associated with $v$. The equalities $\partial_{x}=\partial_{\tau}+\partial_{\overline{\tau}}$ and $\partial_{y}=i(\partial_{\tau}-\partial_{\overline{\tau}})$ allow us to write the first two differential equations also as the vanishing under $4\pi i\partial_{x}-\Delta_{L_{\mathbb{R}}}^{h}$ and $4\pi\partial_{y}-\Delta_{L_{\mathbb{R}},v}^{h}$, as indeed, the actions of $\partial_{x}$ and $\partial_{y}$ multiply the summand associated with $\lambda$ by $\pi i\lambda^{2}$ and $-\pi(\lambda_{v_{+}}^{2}-\lambda_{v_{-}}^{2})$.

\section{Jacobi Forms \label{JacForms}}

We begin with some observations about Fourier expansions.
\begin{lem}
A function of $(\tau,\zeta)\in\mathcal{H} \times L_{\mathbb{C}}$ that, for fixed $\tau\in\mathcal{H}$, is holomorphic and $L$-periodic in $\zeta$, admits a formal Fourier expansion of the form $\sum_{\lambda \in L^{*}}f_{\lambda}(\tau)\mathbf{e}\Big(\tau\frac{\lambda_{v_{+}}^{2}}{2}+\overline{\tau}\frac{\lambda_{v_{-}}^{2}}{2}+(\lambda,\zeta)\Big)$. In case it satisfies the functional equation in $\zeta$ from Proposition \ref{perTheta} (or Equation \eqref{perJac} in Definition \ref{Jacdef} below), the function $f_{\lambda}$ depends only on the coset $\lambda+L \in D_{L}$. \label{Fourinv}
\end{lem}

\begin{proof}
Holomorphicity and $L$-periodicity allow us to expand, for every $\tau$, the function in question as $\sum_{\lambda \in L^{*}}\tilde{f}_{\lambda}(\tau)\mathbf{e}\big((\lambda,\zeta)\big)$. By writing each function $\tilde{f}_{\lambda}(\tau)$ as $\mathbf{e}\Big(\tau\frac{\lambda_{v_{+}}^{2}}{2}+\overline{\tau}\frac{\lambda_{v_{-}}^{2}}{2}\Big)$ times another function $f_{\lambda}(\tau)$, we establish the first assertion. For the second one we fix $\sigma \in L$, consider a function with an expansion as in the first assertion, and substitute this expansion into the two sides of an equality like that from Proposition \ref{perTheta} or Equation \eqref{perJac}. This gives \[\sum_{\lambda \in L^{*}}f_{\lambda}(\tau)\mathbf{e}\bigg(\tau\frac{\lambda_{v_{+}}^{2}}{2}+\overline{\tau}\frac{\lambda_{v_{-}}^{2}}{2}+(\lambda,\zeta)+\tau(\lambda_{v_{+}},\sigma_{v_{+}})+\overline{\tau}(\lambda_{v_{-}},\sigma_{v_{-}})\bigg)\] on the left hand side, while the right hand side becomes \[\sum_{\mu \in L^{*}}f_{\mu}(\tau)\mathbf{e}\bigg(\tau\frac{\mu_{v_{+}}^{2}}{2}+\overline{\tau}\frac{\mu_{v_{-}}^{2}}{2}+(\mu,\zeta)-\tau\frac{\sigma_{v_{+}}^{2}}{2}-\overline{\tau}\frac{\sigma_{v_{-}}^{2}}{2}-(\sigma,\zeta)\bigg).\] Since by substituting $\mu=\lambda+\sigma$ the exponent on the latter expression becomes that from the former, the uniqueness of Fourier expansions yields $f_{\lambda+\sigma}=f_{\lambda}$ as functions of $\tau$, which implies the required assertion. This proves the lemma.
\end{proof}
For presenting the consequence of Lemma \ref{Fourinv} in convenient terms, we recall that the lattice $L(-1)$, in which the signs of the quadratic and bilinear forms are inverted, produces the discriminant group $D_{L(-1)}=D_{L}(-1)$, the basis for which we write as $\{\mathfrak{e}_{\gamma}^{*}\}_{\gamma \in D_{L}}$. Identifying the space $\mathbb{C}[D_{L}(-1)]$ with the dual of $\mathbb{C}[D_{L}]$ by taking the latter basis to be the dual to $\{\mathfrak{e}_{\gamma}\}_{\gamma \in D_{L}}$, we obtain a canonical identification of the representation $\rho_{L(-1)}$ with the representation $\rho_{L}^{*}$ dual to $\rho_{L}$. We denote the resulting bilinear pairing $\mathbb{C}[D_{L}]\times\mathbb{C}[D_{L}(-1)]\to\mathbb{C}$ by $\langle\cdot,\cdot\rangle_{L}$, so that the meaning of this duality is that the equalities
\begin{equation}
\big\langle\rho_{L}(A,\phi)U,V\big\rangle_{L}=\big\langle U,\rho_{L}^{*}(A,\phi)^{-1}V \big\rangle_{L}\mathrm{\ and\ }\big\langle\rho_{L}(A,\phi)U,\rho_{L}^{*}(A,\phi)V\big\rangle_{L}\!=\big\langle U,V \big\rangle_{L} \label{duality}
\end{equation}
hold for every $U\in\mathbb{C}[D_{L}]$, $V\in\mathbb{C}[D_{L}(-1)]$, and $(A,\phi)\in\operatorname{Mp}_{2}(\mathbb{Z})$.

We now obtain the following consequence of Lemma \ref{Fourinv}.
\begin{cor}
Given a smooth function $\Phi:\mathcal{H} \times L_{\mathbb{C}}\to\mathbb{C}$ satisfying the functional equation from Proposition \ref{perTheta}, as in Equation \eqref{perJac} below, there exists a unique smooth function $F:\mathcal{H}\to\mathbb{C}[D_{L}(-1)]$ such that the equality $\Phi(\tau,\zeta)=\big\langle\Theta_{L}(\tau,\zeta;v),F(\tau)\big\rangle_{L}$ holds for every $(\tau,\zeta)\in\mathcal{H} \times L_{\mathbb{C}}$. \label{decomtheta}
\end{cor}

\begin{proof}
Lemma \ref{Fourinv} allows us to define $f_{\gamma}$ for $\gamma \in D_{L}$ to be $f_{\lambda}$ for any $\lambda \in L^{*}$ with $\lambda+L=\gamma$, and we set $F(\tau)=\sum_{\gamma \in D_{L}}f_{\gamma}(\tau)\mathfrak{e}_{\gamma}^{*}$. Then the expansion of $\Phi$ given in the statement of that lemma becomes, using $f_{\gamma}$ for $f_{\lambda}$ with $\lambda+L=\gamma$, just $\sum_{\gamma \in D_{L}}f_{\gamma}(\tau)\sum_{\lambda \in L+\gamma}\mathbf{e}\Big(\tau\frac{\lambda_{v_{+}}^{2}}{2}+\overline{\tau}\frac{\lambda_{v_{-}}^{2}}{2}+(\lambda,\zeta)\Big)$, which amounts to the asserted pairing by Equation \eqref{JacTheta} and the definition of the pairing. The uniqueness of $F$ follows from the fact that its components are determined from the Fourier expansion of $\Phi$ in $\zeta$. This proves the corollary.
\end{proof}

\smallskip

We can now define the main object of this paper.
\begin{defn}
Let $\Gamma$ be a finite index subgroup of $\operatorname{Mp}_{2}(\mathbb{Z})$, and take $k$ and $l$ in $\frac{1}{2}\mathbb{Z}$. Then a \emph{Jacobi form of weight $(k,l)$ and index $(L,v)$ with respect to $\Gamma$} is a smooth function $\Phi:\mathcal{H} \times L_{\mathbb{C}}\to\mathbb{C}$ satisfying the following four conditions: First, the functional equation
\begin{equation}
\Phi(\tau,\zeta+\tau\sigma_{v_{+}}+\overline{\tau}\sigma_{v_{-}}+\nu)=\mathbf{e}\bigg(-\tau\frac{\sigma_{v_{+}}^{2}}{2}-\overline{\tau}\frac{\sigma_{v_{-}}^{2}}{2}-(\sigma,\zeta)\bigg)\Phi(\tau,\zeta) \label{perJac}
\end{equation}
for every $\tau\in\mathcal{H}$, $\zeta \in L_{\mathbb{C}}$, and $\sigma$ and $\nu$ in $L$; Second, the functional equation
\begin{equation}
\Phi\Big(A\tau,\tfrac{\zeta_{v_{+}}}{j(A,\tau)}+\tfrac{\zeta_{v_{-}}}{\overline{j(A,\tau)}}\Big)= \phi(\tau)^{2k}\overline{\phi(\tau)}^{2l}\mathbf{e}\Bigg(\frac{j_{A}'\zeta_{v_{+}}^{2}}{2j(A,\tau)}+\frac{j_{A}'\zeta_{v_{-}}^{2}}{2\overline{j(A,\tau)}}\Bigg)\Phi(\tau,\zeta) \label{modJac}
\end{equation}
for such $\tau$ and $\zeta$ and for $(A,\phi)\in\Gamma$; Third, the function $\zeta\mapsto\Phi(\tau,\zeta)$ is holomorphic for fixed $\tau\in\mathcal{H}$; And fourth, the functions $f_{\lambda}$ from Lemma \ref{Fourinv} have at most linear exponential growth toward each cusp. The Jacobi form is called \emph{pseudo-holomorphic} if it is annihilated by the operator $4\pi i\partial_{\overline{\tau}}-\Delta_{v_{-}}^{h}$ from Proposition \ref{difeq}. It is called \emph{skew-holomorphic} in case it is annihilated by the other operator $4\pi i\partial_{\tau}-\Delta_{v_{+}}^{h}$ from that proposition. \label{Jacdef}
\end{defn}

\begin{rmk}
As in \cite{[BRR]}, \cite{[RR]}, and others, one can consider the Heisenberg group covering $L_{\mathbb{R}}^{2}$ (with kernel $\mathbb{R}$ or $S^{1}=\mathbf{e}(\mathbb{R}/\mathbb{Z})$) and the real Jacobi group in which $\operatorname{SL}_{2}(\mathbb{R})$ (or $\operatorname{Mp}_{2}(\mathbb{R})$) operates on this Heisenberg group. Then Equations \eqref{perJac} and \eqref{modJac} can be viewed as invariance under the corresponding slash operators, when one reduces attention to $\Gamma$ acting on the integral Heisenberg group of $L$ (which is either $L \times L$ or an extension of $L \times L$ by $\mathbb{Z}$). \label{Heis}
\end{rmk}
Note that while the extension to the Heisenberg group from Remark \ref{Heis} has a trivial action in the integral setting, it does not happen in the real one. The resulting factors are related to the functions appearing in Section \ref{Char} below.

We define a \emph{modular form of weight $(k,l)$ and representation $\rho_{L}^{*}$ with respect to $\Gamma$} to be a smooth function $F:\mathcal{H}\to\mathbb{C}[D_{L}(-1)]$ satisfying
\begin{equation}
F(A\tau)=\phi(\tau)^{2k}\overline{\phi(\tau)}^{2l}\rho_{L}^{*}(A,\phi)F(\tau) \label{modeq}
\end{equation}
for every $\tau\in\mathcal{H}$ and $(A,\phi)\in\Gamma$, and having at most linear exponential growth at the cusps. We can now state the first main theorem of this paper.
\begin{thm}
Consider the map taking a Jacobi form $\Phi$ of weight $(k,l)$ and index $(L,v)$ with respect to $\Gamma$, as given in Definition \ref{Jacdef}, to the smooth function $F:\mathcal{H}\to\mathbb{C}[D_{L}(-1)]$ from Corollary \ref{decomtheta}. On the other hand, we define the map taking a modular form of weight $\big(k-\frac{b_{+}}{2},l-\frac{b_{-}}{2}\big)$ and representation $\rho_{L}^{*}$ with respect to $\Gamma$ to the function $\Phi_{L,v}^{F}:\mathcal{H} \times L_{\mathbb{C}}\to\mathbb{C}$ given by the equality $\Phi_{L,v}^{F}(\tau,\zeta):=\big\langle\Theta_{L}(\tau,\zeta;v),F(\tau)\big\rangle_{L}$. Then these maps are inverse isomorphisms between the space of Jacobi forms of such weight and index and the space of modular forms of such weight and representation, both with respect to $\Gamma$. \label{main}
\end{thm}

\begin{proof}
Once we show that the first map takes Jacobi forms to modular forms and the second one takes modular forms to Jacobi forms, it will be clear that the two maps are inverse linear isomorphisms.

Assume thus that $\Phi$ is as in Definition \ref{Jacdef}, and then Equation \eqref{perJac}, Lemma \ref{Fourinv}, and Corollary \ref{decomtheta} indeed produce the unique smooth function $F$ such that $\Phi$ is its pairing with $\Theta_{L}$ from Equation \eqref{JacTheta}. As $F$ has the required growth conditions by Definition \ref{Jacdef}, we need to check its modularity. We therefore fix $\tau\in\mathcal{H}$, $\zeta \in L_{\mathbb{C}}$, and $(A,\phi)\in\operatorname{Mp}_{2}(\mathbb{Z})$, and expand both sides of Equation \eqref{modJac} using the pairing from Corollary \ref{decomtheta}. The right hand side is thus just \[\phi(\tau)^{2k}\overline{\phi(\tau)}^{2l}\mathbf{e}\Bigg(\frac{j_{A}'\zeta_{v_{+}}^{2}}{2j(A,\tau)}+\frac{j_{A}'\zeta_{v_{-}}^{2}}{2\overline{j(A,\tau)}}\Bigg)\big\langle\Theta_{L}(\tau,\zeta;v),F(\tau)\big\rangle_{L},\] while Theorem \ref{modTheta} and the bilinearity of the pairing transform the left hand side of Equation \eqref{modJac} into \[\phi(\tau)^{b_{+}}\overline{\phi(\tau)}^{b_{-}}\mathbf{e}\Bigg(\frac{j_{A}'\zeta_{v_{+}}^{2}}{2j(A,\tau)}+\frac{j_{A}'\zeta_{v_{-}}^{2}}{2\overline{j(A,\tau)}}\Bigg) \big\langle\rho_{L}(A,\phi)\Theta_{L}(\tau,\zeta;v),F(A\tau)\big\rangle_{L}.\] Comparing and applying Equation \eqref{duality}, we deduce that \[\Big\langle\Theta_{L}(\tau,\zeta;v),\phi(\tau)^{2k-b_{+}}\overline{\phi(\tau)}^{2l-b_{-}}F(\tau)\Big\rangle_{L}=\big\langle\Theta_{L}(\tau,\zeta;v),\rho_{L}^{*}(A,\phi)^{-1}F(A\tau)\big\rangle_{L}\] as functions on $\mathcal{H} \times L_{\mathbb{C}}$, from which the uniqueness of Fourier expansions (as in the proof of Corollary \ref{decomtheta}) produces the desired modularity from Equation \eqref{modeq}, with the asserted weights.

Conversely, assume that $F$ satisfies the required modularity condition, and define $\Phi_{L,v}^{F}$ as above. Proposition \ref{perTheta}, the bilinearity, and the fact that $F$ does not depend on $\zeta$ imply that $\Phi=\Phi_{L,v}^{F}$ satisfies Equation \eqref{perJac}. Next, given $(A,\phi)\in\operatorname{Mp}_{2}(\mathbb{Z})$, Theorem \ref{modTheta} and Equation \eqref{modeq} express the left hand side of Equation \eqref{modJac} (with $\Phi=\Phi_{L,v}^{F}$) as \[\mathbf{e}\Bigg(\frac{j_{A}'\zeta_{v_{+}}^{2}}{2j(A,\tau)}+\frac{j_{A}'\zeta_{v_{-}}^{2}}{2\overline{j(A,\tau)}}\Bigg)\phi(\tau)^{2k}\overline{\phi(\tau)}^{2l} \Big\langle\rho_{L}(A,\phi)\Theta_{L}(\tau,\zeta;v),\rho_{L}^{*}(A,\phi)F(\tau)\Big\rangle_{L}.\] As this becomes the right hand side of the latter equation by Equation \eqref{duality} and the definition of $\Phi_{L,v}^{F}$, this functional equation is also established. Since the holomorphicity in $\zeta$ and the growth condition of $F$ are clear, we deduce that $\Phi_{L,v}^{F}$ has the required properties from Definition \ref{Jacdef}. This proves the theorem.
\end{proof}

\begin{rmk}
As the group $\Gamma$ contains some power $T^{N}$ of $T$, for which $j_{T^{N}}'=0$, Equation \eqref{modJac} implies that every Jacobi form $\Phi$ as in Definition \ref{Jacdef} is also $N$-periodic in $\tau$. Then the convenient form for the Fourier expansion from Lemma \ref{Fourinv} is $\sum_{m\in\frac{1}{N}\mathbb{Z}}\sum_{\lambda \in L^{*}}c_{m,\lambda}(y)\mathbf{e}\big(mx+(\lambda,\zeta)\big)e^{-\pi(\lambda_{v_{+}}^{2}-\lambda_{v_{-}}^{2})y}$. The consequence of Equation \eqref{perJac} is the equality $c_{m,\lambda}=c_{m+(\sigma,\lambda)+\frac{\sigma^{2}}{2},\lambda+\sigma}$ (as functions of $y$) for all $\sigma \in L$, which is more familiar from the classical theory of Jacobi forms. It follows that if we define $a_{n,\lambda}$, for $n\in\mathbb{Q}$, to be the function $c_{n+\frac{\lambda^{2}}{2},\lambda}$ (and 0 when the first index is not in $\frac{1}{N}\mathbb{Z}$), then for every $n$ the function $a_{n,\lambda}$ depends only on the image of $\lambda$ in $D_{L}$, hence we can write it as $a_{n,\gamma}$ for $\gamma \in D_{L}$. We can thus express $c_{m,\lambda}$ as $a_{m-\frac{\lambda^{2}}{2},\lambda+L}$, replace $m$ by $n=m-\frac{\lambda^{2}}{2}$, and the summand associated with $\lambda$ and $n$ becomes $a_{n,\lambda+L}\mathbf{e}\big(nx+(\lambda,\zeta)\big)\mathbf{e}\big(x\frac{\lambda^{2}}{2}\big)e^{-\pi(\lambda_{v_{+}}^{2}-\lambda_{v_{-}}^{2})y}$. Since for $\gamma \in D_{L}$ the sum over $\lambda$ with $\lambda+L=\gamma$ is the coefficient $\theta_{L+\gamma}(\tau,\zeta;v)$ appearing in front of $\mathfrak{e}_{\gamma}$ in Equation \eqref{JacTheta}, the remaining expansion, namely $\sum_{\gamma \in D_{L}}\sum_{n\in\mathbb{Q}}\mathbf{e}(nx)a_{n,\gamma}(y)$, is the Fourier expansion of $F$ from Corollary \ref{decomtheta}, i.e., the modular form $F$ such that $\Phi=\Phi_{L,v}^{F}$ as in Theorem \ref{main}. However, for avoiding the question of convergence of such series, it was easier to begin with $F(\tau)$ itself, rather than its Fourier expansion. \label{JacFour}
\end{rmk}

\smallskip

Theorem \ref{main} produces the desired isomorphism on the analytic level. The relation to holomorphicity on the level of vector-valued modular forms is as follows.
\begin{prop}
Given a modular form $F$ of weight $\big(k-\frac{b_{+}}{2},l-\frac{b_{-}}{2}\big)$ and representation $\rho_{L}^{*}$ with respect to $\Gamma$, let $\Phi_{L,v}^{F}$ be the Jacobi form corresponding to it via Theorem \ref{main}. Then $F$ is weakly holomorphic if and only if $\Phi_{L,v}^{F}$ is pseudo-holomorphic. In particular, pseudo-holomorphic Jacobi forms exist only in weights $\big(k-\frac{b_{+}}{2},\frac{b_{-}}{2}\big)$. \label{pshol}
\end{prop}

\begin{proof}
The fact that $F$ does not depend on $\zeta$, the bilinearity of the pairing, and the action of $\partial_{\overline{\tau}}$ on products shows that applying the operator $4\pi i\partial_{\overline{\tau}}-\Delta_{v_{-}}^{h}$ from Definition \ref{Jacdef} to the definition of $\Phi_{L,v}^{F}$ expresses $(4\pi i\partial_{\overline{\tau}}-\Delta_{v_{-}}^{h})\Phi_{L,v}^{F}(\tau,\zeta)$ as \[\big\langle(4\pi i\partial_{\tau}-\Delta_{v_{-}}^{h})\Theta_{L}(\tau,\zeta;v),F(\tau)\big\rangle_{L}+4\pi i\big\langle\Theta_{L}(\tau,\zeta;v),\partial_{\overline{\tau}}F(\tau)\big\rangle_{L}.\] Since the first term vanishes by Proposition \ref{difeq}, the uniqueness of Fourier expansions in $\zeta$ implies that $\Phi_{L,v}^{F}$ is pseudo-holomorphic if and only if $\partial_{\overline{\tau}}F=0$. The fact that $F$ is smooth on $\mathcal{H}$ and has at most linear exponential growth at $\infty$ implies that the latter equality is equivalent to the weak holomorphicity of $F$. The fact that weakly holomorphic modular forms cannot have weights $(k,l)$ with $l\neq0$ (indeed, if $F$ is weakly holomorphic then the left hand side of Equation \eqref{modeq} is holomorphic on $\mathcal{H}$, but the right hand side will not be holomorphic when $l\neq0$) now immediately implies the last assertion. This proves the proposition.
\end{proof}

\begin{rmk}
In terms of the Fourier expansions from Remark \ref{JacFour}, holomorphicity is easier to check if the Fourier expansion of $F$ is in terms of holomorphic exponentials. We therefore write the function $a_{n,\gamma}(y)$ from that expansion as $\tilde{a}_{n,\gamma}(y)e^{-2\pi ny}$ (so that it indeed multiplies $\mathbf{e}(n\tau)$), which means that the function $c_{m,\lambda}(y)=a_{m-\frac{\lambda^{2}}{2},\lambda+L}(y)$ is $e^{-2\pi my+\pi\lambda^{2}y}$ times $\tilde{c}_{m,\lambda}(y)=\tilde{a}_{m-\frac{\lambda^{2}}{2},\lambda+L}(y)$. Recalling that $\lambda^{2}=\lambda_{v_{+}}^{2}+\lambda_{v_{-}}^{2}$, the corresponding expansion of $\Phi_{L,v}^{F}(\tau,\zeta)$ is as $\sum_{m\in\frac{1}{N}\mathbb{Z}}\sum_{\lambda \in L^{*}}\tilde{c}_{m,\lambda}(y)\mathbf{e}\big(m\tau+(\lambda,\zeta)\big)e^{2\pi\lambda_{v_{-}}^{2}y}$. This makes Proposition \ref{pshol} a bit more visible, since $F$ is weakly holomorphic if and only if the functions $\tilde{c}_{m,\lambda}$ are constant. Indeed, this is equivalent to the operator $4\pi i\partial_{\overline{\tau}}$ from Definition \ref{Jacdef} operating only on the latter exponentials, in the same way that $\Delta_{v_{-}}^{h}$ acts on the exponential of $(\lambda,\zeta)$. \label{holFour}
\end{rmk}
Note that the last formula in Remark \ref{holFour} is in correspondence with $F$ and $\Phi$ being holomorphic together in the positive definite case, since then the exponentials in $y$ all become trivial.

\smallskip

We conclude with the following remark.
\begin{rmk}
In the indefinite case, we obtain the isomorphisms from Theorem \ref{main} and Proposition \ref{pshol} for every $v\in\operatorname{Gr}(L_{\mathbb{R}})$. Therefore, for a fixed modular form $F$, we obtain a function on $\mathcal{H} \times L_{\mathbb{C}}\times\operatorname{Gr}(L_{\mathbb{R}})$, namely the one taking $\tau$, $\zeta$, and $v$ to $\Phi_{L,v}^{F}(\tau,\zeta)$ (perhaps better written as $\Phi_{L}^{F}(\tau,\zeta;v)$, or just $\Phi(\tau,\zeta;v)$, for this point of view). It is clear that this function is smooth, but it also has an additional invariance property. Recall that $\Gamma_{L}\subseteq\operatorname{SO}^{+}(L)$ operates on $\operatorname{Gr}(L_{\mathbb{R}})$, and by extending scalars from $L$ or $L_{\mathbb{R}}$ to $L_{\mathbb{C}}$ it operates both on the variables $v$ and $\zeta$. Since one easily checks that the Jacobi--Siegel theta function from Equation \eqref{JacTheta} is invariant under the diagonal action of the discriminant kernel $\Gamma_{L}$ from Equation \eqref{GammaL} on these two variables, we deduce that the equality $\Phi(\tau,\zeta;v)=\Phi(\tau,\mathcal{A}\zeta;\mathcal{A}v)$ holds for every such variables and $\mathcal{A}$ in the latter group (in particular, in the definite case, all our Jacobi forms are invariant under the action of the finite group $\Gamma_{L}$ on the variable $\zeta$). \label{funcofv}
\end{rmk}
It will be interesting to see what other properties do the functions from Remark \ref{funcofv} have.

\section{Some Operations \label{Oper}}

It is clear that the Grassmannians associated with $L$ and with $L(-1)$ are canonically isomorphic (with $v_{+}$ and $v_{-}$ interchanged for every $v$ in the common Grassmannian). Moreover, if complex conjugation is defined to take $\mathfrak{e}_{\gamma}\in\mathbb{C}[D_{L}]$ to $\mathfrak{e}_{\gamma}^{*}\in\mathbb{C}[D_{L}(-1)]$, then we get the canonical identification of $\rho_{L(-1)}=\rho_{L}^{*}$ with the complex conjugate representation $\overline{\rho}_{L}$. This shows that for $U\in\mathbb{C}[D_{L}]$ and $V\in\mathbb{C}[D_{L}(-1)]$ we have $\overline{\langle U,V \rangle_{L}}=\langle\overline{U},\overline{V}\rangle_{L(-1)}$. It also follows that if $F$ is a modular form of weight $\big(k-\frac{b_{+}}{2},l-\frac{b_{-}}{2}\big)$ and representation $\rho_{L}^{*}$ then the complex conjugate function $\overline{F}$ is modular of weight $\big(l-\frac{b_{-}}{2},k-\frac{b_{+}}{2}\big)$ and representation $\rho_{L}$. Let $\Phi_{L,v}^{F}$ be the Jacobi form of weight $(k,l)$ and index $(L,v)$ that is associated with $F$ by Theorem \ref{main}, and similarly $\Phi_{L(-1),v}^{\overline{F}}$ is the Jacobi form of weight $(l,k)$ index $\big(L(-1),v\big)$ corresponding to $\overline{F}$. Then conjugating Equation \eqref{JacTheta} and the formula from Corollary \ref{decomtheta} yields the equalities \[\Theta_{L(-1)}(\tau,\zeta;v)=\overline{\Theta_{L}(\tau,\overline{\zeta};v)}\qquad\mathrm{and}\qquad\Phi_{L(-1),v}^{\overline{F}}(\tau,\zeta)=\overline{\Phi_{L,v}^{F}(\tau,\overline{\zeta})}\] (note that the double conjugation on $\zeta$ is in correspondence with the requirement that all our Jacobi forms, including the Jacobi--Siegel theta functions, be holomorphic in this variable).

However, the relation to skew-holomorphic Jacobi forms, as defined in Definition \ref{Jacdef} (the special cases from the literature will be explained below), is not obtained by complex conjugation of the modular forms. For describing it we need the following lemma.
\begin{lem}
For $(A,\phi)\in\operatorname{Mp}_{2}(\mathbb{R})$, with $A=\binom{a\ \ b}{c\ \ d}$, set $\tilde{A}=\binom{\ \ a\ \ -b}{-c\ \ d}$, and let $\tilde{\phi}$ be the function $\tau\mapsto\overline{\phi(-\overline{\tau})}$. Then the map $(A,\phi)\mapsto(\tilde{A},\tilde{\phi})$ is a group involution of $\operatorname{Mp}_{2}(\mathbb{R})$. Moreover, if $\omega_{L}:\mathbb{C}[D_{L}]\to\mathbb{C}[D_{L}(-1)]$ is the $\mathbb{C}$-linear map sending $\mathfrak{e}_{\gamma}$ to $\mathfrak{e}_{\gamma}^{*}$ for every $\gamma \in D_{L}$, then the equality $\rho_{L}^{*}(A,\phi)\circ\omega_{L}=\omega_{L}\circ\rho_{L}(\tilde{A},\tilde{\phi})$ holds for every $(A,\phi)\in\operatorname{Mp}_{2}(\mathbb{Z})$. \label{mapreps}
\end{lem}

\begin{proof}
The first assertion follows from the fact that the definition of $\operatorname{Mp}_{2}(\mathbb{R})$ in Equation \eqref{Mp2R} can be extended to a double cover of $\operatorname{GL}_{2}(\mathbb{R})$, and the operation in question is just conjugation by the element $\big(\binom{-1\ \ 0}{\ \ 0\ \ 1},1\big)$, which is of order 2 and acts as $\tau\mapsto-\overline{\tau}$ on $\mathcal{H}$, inside this larger group. Next, we note that $\tilde{T}=T^{-1}$ and $\tilde{S}=S^{-1}$, and then the equality for $T$ holds because $\mathfrak{e}_{\gamma}$ and $\mathfrak{e}_{\gamma}^{*}$ are eigenvectors of $\rho_{L}(T)$ and $\rho_{L}^{*}(T)$, with multiplicative inverse eigenvalues. As the formula for $S$ is verified by a direct comparison, and these elements generate $\operatorname{Mp}_{2}(\mathbb{Z})$, the assertion follows for every element of that group. This proves the lemma.
\end{proof}
We remark that the relation $(ST)^{3}=S^{2}$ defining the braid group covering $\operatorname{SL}_{2}(\mathbb{Z})$ and $\operatorname{Mp}_{2}(\mathbb{Z})$, which is equivalent to $(TS)^{3}=S^{2}$ by conjugation, is preserved under taking the generators $S$ and $T$ to their inverses, as in the proof of Lemma \ref{mapreps}. The last equality in that lemma may also be verified using the formulae from \cite{[Sch]}, \cite{[Str]}, and \cite{[Ze1]}, but our argument seems to be shorter.

\begin{cor}
For a subgroup $\Gamma$ of finite index in $\operatorname{Mp}_{2}(\mathbb{Z})$, let $\tilde{\Gamma}$ be its image under the involution from Lemma \ref{mapreps}, and take a modular form $G$ of weight $(k,l)$ and representation $\rho_{L}$ with respect to $\tilde{\Gamma}$. Then the function defined by $\tilde{G}(\tau)=\omega_{L}\big(G(-\overline{\tau})\big)$ is a modular form of weight $(l,k)$ and representation $\rho_{L}^{*}$ with respect to $\Gamma$. \label{mapMFs}
\end{cor}

\begin{proof}
We have the equality $-\overline{A\tau}=\tilde{A}(-\overline{\tau})$ using the map from Lemma \ref{mapreps} (this is essentially the definition using the conjugation in the extended group in the proof of that lemma), and since $(A,\phi)\in\Gamma$ if and only if $(\tilde{A},\tilde{\phi})\in\tilde{\Gamma}$, Equation \eqref{modeq} implies that \[\tilde{G}(A\tau)=\omega_{L}\big[G\big(-\overline{A\tau})\big]\!=\!\omega_{L}\big[G\big(\tilde{A}(-\overline{\tau})\big)\big]\!= \!\omega_{L}\Big[\rho_{L}(\tilde{A},\tilde{\phi})\tilde{\phi}(-\overline{\tau})^{2k}\overline{\tilde{\phi}(-\overline{\tau})}^{2l}G(-\overline{\tau})\Big]\!.\] But using the last equality from that lemma and the fact that $\tilde{\phi}(-\overline{\tau})=\overline{\phi(\tau)}$ by definition, the right hand side here is $\phi(\tau)^{2l}\overline{\phi(\tau)}^{2k}$ times $\rho_{L}^{*}(A,\phi)$ acting on $\omega_{L}\big(G(-\overline{\tau})\big)=\tilde{G}(\tau)$, thus yielding the corresponding Equation \eqref{modeq}. This proves the corollary.
\end{proof}

In fact, the extension of Equation \eqref{Mp2R} sometimes involves complex conjugation in the definition of $\phi$ for matrices of negative determinant (for their action to preserve $\mathcal{H}$), which is in correspondence with the picture arising from Corollary \ref{mapMFs}.

\smallskip

We can now prove the following result.
\begin{prop}
Let $G$ be a modular form of weight $\big(k-\frac{b_{-}}{2},l-\frac{b_{+}}{2}\big)$ and representation $\rho_{L}$ with respect to the group $\tilde{\Gamma}$ from Corollary \ref{mapMFs}. Then if $\tilde{G}$ is the modular form from that corollary, write $\tilde{\Phi}_{L,v}^{G}$ for the Jacobi form $\Phi_{L,v}^{\tilde{G}}$, which is of weight $(l,k)$ and index $(L,v)$ with respect to $\Gamma$. Then the map taking $G$ to $\tilde{\Phi}_{L,v}^{G}$ is an isomorphism between the space of modular forms of such weight and representation and that of Jacobi forms of that weight and index, and $\tilde{\Phi}_{L,v}^{G}$ is skew-holomorphic if and only if $G$ is weakly holomorphic, a case which can occur only if $l=\frac{b_{+}}{2}$. \label{skewhol}
\end{prop}
The case of positive definite $L$, with $b_{-}=0$, indeed reproduces the definition from \cite{[H]}, and in case $b_{+}=1$ we obtain the skew-holomorphic Jacobi forms from \cite{[Sk]}, \cite{[BR]}, and others. Note that the transformation formula of Skew-holomorphic Jacobi forms of weight $k$ in \cite{[Sk]} and \cite{[BR]} involves the factor $j^{1/2}\overline{j}^{k-1/2}$, so in our convention its weight is $\big(\frac{1}{2},k-\frac{1}{2}\big)$, whence the chosen notation for the weight in Proposition \ref{skewhol}.

\begin{proof}
Since the map $G\mapsto\tilde{G}$ from Corollary \ref{mapMFs} is clearly an isomorphism between the vector spaces of modular forms appearing there, the fact that the map $G\mapsto\tilde{\Phi}_{L,v}^{G}$ is an isomorphism follows directly from Theorem \ref{main}. Next, as in the proof of Proposition \ref{pshol}, we can write $(4\pi i\partial_{\tau}-\Delta_{v_{+}}^{h})\tilde{\Phi}_{L,v}^{G}(\tau,\zeta)$ as \[\big\langle(4\pi i\partial_{\tau}-\Delta_{v_{+}}^{h})\Theta_{L}(\tau,\zeta;v),\tilde{G}(\tau)\big\rangle_{L}+4\pi i\big\langle\Theta_{L}(\tau,\zeta;v),\partial_{\tau}\tilde{G}(\tau)\big\rangle_{L},\] and the first term again vanishes by Proposition \ref{difeq}. Therefore the skew-holomorphicity of $\tilde{\Phi}_{L,v}^{G}$ is equivalent to $\tilde{G}$ being annihilated by $\partial_{\tau}$, i.e., being anti-holomorphic. But the definition of $\tilde{G}$ and the growth conditions in Definition \ref{Jacdef} again imply that this is the case if and only if $G$ is weakly holomorphic, and the restriction on $l$ is again as in Proposition \ref{pshol}. This proves the proposition.
\end{proof}

\smallskip

Let $\Lambda$ be a finite index sub-lattice of $L$, a case in which $L$ is called an \emph{over-lattice} of $\Lambda$, and $D_{L}$ is the quotient $H^{\perp}/H$ for the isotropic subgroup $H=L/\Lambda$ of $D_{\Lambda}$. Then the Grassmannians $\operatorname{Gr}(L_{\mathbb{R}})$ and $\operatorname{Gr}(\Lambda_{\mathbb{R}})$ are clearly the same. Recall from \cite{[Ma]} and others that there are two maps $\uparrow^{L}_{\Lambda}:\mathbb{C}[D_{L}]\to\mathbb{C}[D_{\Lambda}]$ and $\downarrow^{L}_{\Lambda}:\mathbb{C}[D_{\Lambda}]\to\mathbb{C}[D_{L}]$, that are defined by \[\uparrow^{L}_{\Lambda}\mathfrak{e}_{\gamma}:=\sum_{\delta \in H^{\perp},\ \delta+H=\gamma}\mathfrak{e}_{\delta}\qquad\mathrm{and}\qquad\downarrow^{L}_{\Lambda}\mathfrak{e}_{\delta}:=\begin{cases} \mathfrak{e}_{\delta+H}, & \mathrm{in\ case\ }\delta \in H^{\perp}, \\  0, & \mathrm{when\ }\delta \not\in H^{\perp}. \end{cases}\] respectively. Viewing $\mathbb{C}[D_{L}]$ and $\mathbb{C}[D_{\Lambda}]$ as the representation spaces of $\rho_{L}$ and $\rho_{\Lambda}$ respectively, Lemma 2.1 and Corollary 2.2 of \cite{[Ma]} show that both maps are maps of $\operatorname{Mp}_{2}(\mathbb{Z})$-representations, and they thus take modular forms with each representation (with respect to any subgroup of $\operatorname{Mp}_{2}(\mathbb{Z})$) to modular forms with the other representation, with the same weight. One relation between these operators and the maps from Theorem \ref{main} and Proposition \ref{skewhol} is the following one.
\begin{prop}
Let $F$ and $G$ be modular forms with the parameters from Theorem \ref{main} and Proposition \ref{skewhol} respectively, and consider the lattice $L$ as an over-lattice of the lattice $\Lambda$. We then have the equalities \[\Phi_{\Lambda,v}^{\uparrow^{L(-1)}_{\Lambda(-1)}F}=\Phi_{L,v}^{F}\qquad\mathrm{and}\qquad\tilde{\Phi}_{\Lambda,v}^{\uparrow^{L}_{\Lambda}G}=\tilde{\Phi}_{L,v}^{G}.\] \label{uparrow}
\end{prop}

\begin{proof}
We first observe that the Jacobi--Siegel theta functions from Equation \eqref{JacTheta} satisfy the equality $\Theta_{L}(\tau,\zeta;v)=\downarrow^{L}_{\Lambda}\Theta_{\Lambda}(\tau,\zeta;v)$ (Lemma 1.5 of \cite{[Ze3]} shows it for the theta functions from Equation \eqref{Thetadef}, and the proof works here exactly in the same manner). Therefore the first equality follows directly from the well-known relation $\big\langle U,\uparrow^{L(-1)}_{\Lambda(-1)}V \big\rangle_{L}=\big\langle\downarrow^{L}_{\Lambda}U,V\big\rangle_{L}$ for $U\in\mathbb{C}[D_{L}]$ and $V\in\mathbb{C}[D_{\Lambda(-1)}]$ (see, e.g., Lemma 2.2 of \cite{[Ze3]}). Now, we have the immediate commutation relation $\omega_{\Lambda}\circ\uparrow^{L}_{\Lambda}=\uparrow^{L(-1)}_{\Lambda(-1)}\circ\omega_{L}$, and therefore if $\tilde{G}$ is the modular form associated with $G$ in Corollary \ref{mapMFs}, then the one associated with $\uparrow^{L}_{\Lambda}G$ is $\uparrow^{L(-1)}_{\Lambda(-1)}\tilde{G}$ (changing the variable to $-\overline{\tau}$ is not affected by these operations). Therefore the second equality follows from the first. This proves the proposition.
\end{proof}
The relation involving $\downarrow^{L(-1)}_{\Lambda(-1)}F$ (or $\downarrow^{L}_{\Lambda}G$) will require theta functions with characteristics, and will be proved in Proposition \ref{downarrow} below.

\smallskip

Consider now another lattice $K$, of signature $(a_{+},a_{-})$, and the orthogonal direct sum $L \oplus K$, of signature $(a_{+}+b_{+},a_{-}+b_{-})$. Then we have the equalities $D_{L \oplus K}=D_{L} \oplus D_{K}$, $\mathbb{C}[D_{L \oplus K}]=\mathbb{C}[D_{L}]\otimes\mathbb{C}[D_{K}]$, and $\rho_{L \oplus K}=\rho_{L}\otimes\rho_{K}$, and analogous identities after dualizing. Given elements $v\in\operatorname{Gr}(L_{\mathbb{R}})$ and $w\in\operatorname{Gr}(K_{\mathbb{R}})$, the pair of spaces $v_{+} \oplus w_{+}$ and $v_{-} \oplus w_{-}$ represent an element of $\operatorname{Gr}(L_{\mathbb{R}} \oplus K_{\mathbb{R}})$, that we denote, following \cite{[Ze3]}, by $v \oplus w$. If $F$ is a modular form of weight $(k,l)$ and representation $\rho_{L}^{*}$ with respect to a subgroup $\Gamma$ of $\operatorname{Mp}_{2}(\mathbb{Z})$, and $H$ is a modular form of some other weight $(\kappa,\mu)$ and representation $\rho_{K}^{*}$ with respect to the same group $\Gamma$, then $F \otimes H$ is a modular form of weight $(k+\kappa,l+\mu)$ and representation $\rho_{L \oplus K}^{*}$ with respect to $\Gamma$. We then obtain the following result.
\begin{prop}
With $\tau\in\mathcal{H}$, and writing elements of $(L \oplus K)_{\mathbb{C}}$ as sums of pairs of elements $\zeta \in L_{\mathbb{C}}$ and $\xi \in K_{\mathbb{C}}$, we obtain the equality \[\Phi_{L \oplus K,v \oplus w}^{F \otimes H}(\tau,\zeta+\xi)=\Phi_{L,v}^{F}(\tau,\zeta)\cdot\Phi_{K,w}^{H}(\tau,\xi).\] \label{dirsum}
\end{prop}

\begin{proof}
The Jacobi--Siegel theta functions from Equation \eqref{JacTheta} satisfy \[\Theta_{L \oplus K}(\tau,\zeta+\xi;v \oplus w)=\Theta_{L}(\tau,\zeta;v)\otimes\Theta_{K}(\tau,\xi;w)\] (as a simple modification of the proof of the analogous statement in Lemma 1.2 of \cite{[Ze3]} for the functions from Equation \eqref{Thetadef} shows). The result now follows from the definition of these Jacobi forms in Theorem \ref{main}, since given $U\in\mathbb{C}[D_{L}]$, $V\in\mathbb{C}[D_{\Lambda(-1)}]$, $X\in\mathbb{C}[D_{K}]$, and $Y\in\mathbb{C}[D_{K(-1)}]$ we clearly have the equality $\langle U \otimes X,V \otimes Y \rangle_{L \oplus K}=\langle U,V \rangle_{L}\cdot\langle X,Y \rangle_{K}$ (indeed, for $X=\mathfrak{e}_{\gamma}$, $Y=\mathfrak{e}_{\delta}^{*}$, $U=\mathfrak{e}_{\alpha}$, $Y=\mathfrak{e}_{\beta}^{*}$ both sides equal 1 when $\gamma=\delta$ and $\alpha=\beta$ and 0 otherwise, and extend by linearity). This proves the proposition.
\end{proof}

\begin{rmk}
If the lattice $K$ is the trivial lattice, of rank 0, then $\rho_{K}^{*}$ is trivial, and so is the theta function $\Theta_{K}$ (also the Jacobi--Siegel one, since the variable $\xi$ in the corresponding Equation \eqref{JacTheta} is taken from the trivial space $K_{\mathbb{C}}$). Then both maps in Theorem \ref{main} are the identity map on the space of scalar-valued modular forms $h$ (the Jacobi forms are also modular forms, since $\xi$ is trivial). Therefore this special case of Proposition \ref{dirsum} reduces to the statement that the maps from Theorem \ref{main}, when combined over all the possible weights, are maps of algebras over the ring of scalar-valued modular forms with respect to $\Gamma$ (with the analytic properties that one wishes to consider). \label{algscal}
\end{rmk}

\smallskip

Consider now the case where $M$ is a primitive non-degenerate sub-lattice of $L$, of some signature $(c_{+},c_{-})$. Denote the intersection of the space $M^{\perp}$ with $L$ by $M^{\perp}_{L}$, and assume that the element $v\in\operatorname{Gr}(L_{\mathbb{R}})$ is the direct sum $u \oplus u^{\perp}$ for $u\in\operatorname{Gr}(M_{\mathbb{R}})$ and $u^{\perp}\in\operatorname{Gr}(M_{\mathbb{R}}^{\perp})$. Then every $\zeta \in L_{\mathbb{C}}$ is the orthogonal sum of the projections $\zeta_{M_{\mathbb{C}}}$ and $\zeta_{M_{\mathbb{C}}^{\perp}}$, and $\zeta_{v_{\pm}}$ is the orthogonal direct sum of $\zeta_{u_{\pm}}$ and $\zeta_{u^{\perp}_{\pm}}$ for each sign $\pm$. In analogy with Proposition 1.6 of \cite{[Ze3]}, we have the equality \[\Theta_{L}(\tau,\zeta;v)=\Big\downarrow^{L}_{M \oplus M^{\perp}_{L}}\Big[\Theta_{M}(\tau,\zeta_{M_{\mathbb{C}}};u)\otimes\Theta_{M^{\perp}}\big(\tau,\zeta_{M_{\mathbb{C}}^{\perp}};u^{\perp}\big)\Big]\] (because we have already established the analogues of Lemmas 1.2 and 1.5 from that reference for our Jacobi--Siegel theta functions).

However, there is another expression for $\Theta_{L}(\tau,\zeta;v)$ in this case. Denote by $\pi_{M}$ the composition of the surjective projections from $L^{*}$ onto $M^{*}$ and from $M^{*}$ onto $D_{M}$, and observe that the projection onto $M_{\mathbb{R}}^{\perp}$, hence also onto $u^{\perp}_{\pm}$, are well-defined for elements of $L^{*}/M$. Following Equation (8) of \cite{[Ze3]} we define, for $\eta \in M_{\mathbb{C}}^{\perp}$ and $\tau$ and $u^{\perp}$ as above, the function
\begin{equation}
\Theta_{L,M}(\tau;\eta;u^{\perp}):=\sum_{\delta \in D_{M}}\sum_{\substack{\lambda \in L^{*}/M \\ \pi_{M}(\lambda)=\delta}} \mathbf{e}\Bigg(\tau\frac{\lambda_{u^{\perp}_{+}}^{2}}{2}+\overline{\tau}\frac{\lambda_{u^{\perp}_{-}}^{2}}{2}+\big(\lambda_{M_{\mathbb{R}}^{\perp}},\eta\big)\Bigg)\mathfrak{e}_{\lambda+L}\otimes\mathfrak{e}_{\delta}^{*} \label{compTheta}
\end{equation}
with values in $\mathbb{C}[D_{L}]\otimes\mathbb{C}[D_{M(-1)}]$, and obtain the following result.
\begin{thm}
The theta function $\Theta_{L,M}$ from Equation \eqref{compTheta} satisfies the functional equation from Proposition \ref{perTheta} for $\sigma$ and $\nu$ from $M^{\perp}_{L}$, as well as that from Theorem \ref{modTheta}, with the weight being $\big(\frac{b_{+}-c_{+}}{2},\frac{b_{-}-c_{-}}{2}\big)$ and the representation $\rho_{L}\otimes\rho_{M}^{*}$. In particular, the function $\tau\mapsto\Theta_{L,M}(\tau;u^{\perp}):=\Theta_{L,M}(\tau,0;u^{\perp})$ is a modular form of weight $\big(\frac{b_{+}-c_{+}}{2},\frac{b_{-}-c_{-}}{2}\big)$ and representation $\rho_{L}\otimes\rho_{M}^{*}$. In addition, for every $\tau\in\mathcal{H}$ and $\zeta \in L_{\mathbb{C}}$ we have the equality \[\Theta_{L}(\tau,\zeta;v)=\big\langle\Theta_{M}(\tau,\zeta_{M_{\mathbb{C}}};u),\Theta_{L,M}\big(\tau,\zeta_{M_{\mathbb{C}}^{\perp}};u^{\perp}\big)\big\rangle_{M}.\] \label{modThetaLM}
\end{thm}

\begin{proof}
The third assertion follows from the proof of Proposition 1.8 of \cite{[Ze3]}, \emph{mutatis mutandis}. Now, the analogue of Lemma 1.9 of that reference presents $\Theta_{L,M}$, in the case where $L=M \oplus M^{\perp}_{L}$, as $\Theta_{M^{\perp}_{L}}$ (with the same variables), to which Proposition \ref{perTheta} and Theorem \ref{modTheta} give the desired properties, tensored with $\sum_{\delta \in D_{M}}\mathfrak{e}_{\delta}\otimes\mathfrak{e}_{\delta}^{*}$. In the general case the presentation from Lemma 1.10 of that reference, using the operator $\downarrow^{L \oplus M(-1)}_{\Lambda \oplus M(-1)}$ for $\Lambda=M \oplus M^{\perp}_{L}$, is valid here as well. Thus the first assertion is proved like Theorem 1.11 there (where the analogue of Lemma 1.4 of that reference is not required for the equation from Proposition \ref{perTheta} and is replaced by $\downarrow^{L}_{\Lambda}$ being a map of representations for that from Theorem \ref{modTheta}), and the second one is an immediate consequence of the first (or of the latter theorem). This proves the theorem.
\end{proof}

The paper \cite{[Ma]} defines an operation called theta contraction, which was generalized in \cite{[Ze3]} to the following notion. For a modular form $F$ of weight $(k,l)$ and representation $\rho_{L}^{*}$, a primitive non-degenerate sub-lattice $M$ of $L$, and an element $v\in\operatorname{Gr}(L_{\mathbb{R}})$ which is the direct sum of $u\in\operatorname{Gr}(M_{\mathbb{R}})$ and $u^{\perp}\in\operatorname{Gr}(M_{\mathbb{R}^{\perp}})$, we define the \emph{theta contraction of $F$ with respect to $u^{\perp}$} to be
\begin{equation}
\Theta_{(L,M)}(F;u^{\perp}):\mathcal{H}\to\mathbb{C}[D_{M(-1)}],\ \ \ \Theta_{(L,M)}(F;u^{\perp}):\tau\mapsto\big\langle\Theta_{L,M}(\tau;u^{\perp}),F(\tau)\big\rangle_{L}. \label{contdef}
\end{equation}
The second assertion of Theorem \ref{modThetaLM} implies that $\Theta_{(L,M)}(F;u^{\perp})$ is a modular form of weight $\big(k+\frac{b_{+}-c_{+}}{2},l+\frac{b_{-}-c_{-}}{2}\big)$ and representation $\rho_{M}^{*}$. We then obtain the following result.
\begin{prop}
Let $\Phi$ be a Jacobi form of weight $(k,l)$ and index $(L,v)$ with respect to $\Gamma$, and decompose $L_{\mathbb{C}}$ as $M_{\mathbb{C}} \oplus M_{\mathbb{C}}^{\perp}$. Then the map \[\operatorname{Res}^{L}_{M}\Phi:\mathcal{H} \times M_{\mathbb{C}}\to\mathbb{C},\qquad\operatorname{Res}^{L}_{M}\Phi(\tau,\xi)=\Phi(\tau,\xi+0)\quad\mathrm{with}\quad0 \in M_{\mathbb{C}}^{\perp}\] is a Jacobi form of weight $(k,l)$ and index $(M,u)$ with respect to $\Gamma$. Moreover, if we write $\Phi$ as $\Phi_{L,v}^{F}$ for $F$ as in Theorem \ref{main} then $\operatorname{Res}^{L}_{M}\Phi$ is $\Phi_{M,u}^{\Theta_{(L,M)}(F;u^{\perp})}$. \label{Jaccont}
\end{prop}

\begin{proof}
Restricting Equation \eqref{perJac} for $\Phi$ to $\sigma$ and $\nu$ in $M$ and recalling that $\sigma_{v_{\pm}}=\sigma_{u_{\pm}}$ and $(\sigma,\xi)=(\sigma,\xi+0)$ for such $\sigma$ yields the same equation for $\operatorname{Res}^{L}_{M}\Phi$. Equation \eqref{modJac} also transforms easily, using the fact that the $M_{\mathbb{C}}^{\perp}$-part of $\xi+0$ vanishes. Since the resulting combinations of the Fourier coefficients of $\Phi$, which grow at most linearly exponentially, also grows at most linearly exponentially, this proves the first assertion. For the second one, write $\Phi=\Phi_{L,v}^{F}$ as in Corollary \ref{decomtheta} and Theorem \ref{main}, and express $\theta_{L}(\tau,\xi+0;v)$ as in Theorem \ref{modThetaLM}. But now Lemma 2.2 of \cite{[Ze3]} expresses the resulting pairing as $\big\langle\Theta_{M}(\tau,\xi;u),\langle\Theta_{L,M}(\tau,0;u^{\perp}),F(\tau)\rangle_{L}\big\rangle_{M}$, and as in the proof of Theorem 2.6 of that reference, the internal pairing is $\Theta_{(L,M)}(F;u^{\perp})$ by Equation \eqref{contdef}. We thus obtain the desired result by Theorem \ref{main} again. This proves the proposition.
\end{proof}
Note that for a general element $\zeta \in L_{\mathbb{C}}$, namely $\zeta=\xi+\eta$ for such $\xi$ and some $\eta \in M_{\mathbb{C}}^{\perp}$, the proof of Proposition \ref{Jaccont} expresses $\Phi_{L,v}^{F}(\tau,\xi+\eta;v)$ as the $M$-pairing of $\Theta_{M}(\tau,\xi;u)$ with the function $\big\langle\Theta_{L,M}(\tau,\eta;u^{\perp}),F(\tau)\big\rangle_{L}$, which we may consider as a generalized theta contraction. The latter, viewed as a function on $\mathcal{H} \times M_{\mathbb{C}}^{\perp}$, is a (vector-valued) Jacobi form, of weight $\big(k+\frac{b_{+}-c_{+}}{2},l+\frac{b_{-}-c_{-}}{2}\big)$, index $(M^{\perp}_{L},u^{\perp})$, and representation $\rho_{M}^{*}$. It will be interesting to see which properties do such functions have, especially in relation to $\Phi_{L,v}^{F}(\tau,\xi+\eta;v)$.

We recall that in the classical case, where the index is a number (rather than a lattice), the product of Jacobi forms is a Jacobi form, where both the weights and the index are added. With lattices one cannot expect such a formula in general, since the variable $\zeta$ is taken from two different spaces (sometimes in different dimensions). In the positive definite lattice index case, a formula can be obtained by representing lattices using Gram matrices. However, working in the terminology of abstract lattices, and extending to our more general situation, this definition takes the following form.
\begin{prop}
Let $L$, $K$, $v$, and $w$ be as in Proposition \ref{dirsum}. Assume that there is a group isomorphism $\iota:L \to K$ such that its extension $\iota_{\mathbb{R}}:L_{\mathbb{R}} \to K_{\mathbb{R}}$ takes each of the spaces $v_{\pm}$ onto the respective space $w_{\pm}$, so that in particular $a_{+}=b_{+}$ and $a_{-}=b_{-}$. We define $L+_{\iota}K$ to be $L$ as a group, endowed with the quadratic form $\lambda\mapsto\frac{\lambda^{2}}{2}+\frac{(\iota\lambda)^{2}}{2}$. Let $v+_{\iota}w$ be the vector space decomposition of $(L+_{\iota}K)_{\mathbb{R}}=L_{\mathbb{R}}$ as $v_{+} \oplus v_{-}$. Then $L+_{\iota}K$ is a (non-degenerate) lattice of the same signature $(b_{+},b_{-})$ as $L$ and $K$, $v+_{\iota}w$ represents an element of $\operatorname{Gr}\big((L+_{\iota}K)_{\mathbb{R}}\big)$, and for a Jacobi form $\Phi$ of weight $(k,l)$ and index $(L,v)$ and a Jacobi form $\Psi$ of weight $(\kappa,\mu)$ and index $(K,w)$, the map $\Phi\Psi$ taking $\tau\in\mathcal{H}$ and $\zeta \in L_{\mathbb{C}}$ to $\Phi(\tau,\zeta)\Psi(\tau,\iota_{\mathbb{C}}\zeta)$ is a Jacobi form of weight $(k+\kappa,l+\mu)$ and index $(L+_{\iota}K,v+_{\iota}w)$. \label{prod}
\end{prop}

\begin{proof}
Checking non-degeneracy and the signature can be carried out on the level of the real quadratic spaces, where it is clear from the assumption that the $(L+_{\iota}K)_{\mathbb{R}}$-norm of any non-zero element of $(v+_{\iota}w)_{+}=v_{+}$ (resp. $(v+_{\iota}w)_{-}=v_{-}$) is positive (resp. negative). Moreover, the $(L+_{\iota}K)_{\mathbb{R}}$-pairing of two elements there are the sum of their $L_{\mathbb{R}}$-pairings and the $K_{\mathbb{R}}$-pairing of their $\iota_{\mathbb{R}}$-images, so that if one is in $(v+_{\iota}w)_{+}$ and the other is in $(v+_{\iota}w)_{-}$ then the pairing vanishes. Therefore $(L+_{\iota}K)_{\mathbb{R}}$ is indeed the orthogonal direct sum of the $b_{+}$-dimensional positive definite space $(v+_{\iota}w)_{+}$ and the $b_{-}$-dimensional positive definite space $(v+_{\iota}w)_{-}$, yielding the non-degeneracy, the signature, and the identification of $v+_{\iota}w$ as a Grassmannian element. Moreover, applying Equations \eqref{perJac} and \eqref{modJac} for $\Phi$ and $\Psi$ and the same considerations produce these equalities for the product $\Phi\Psi$, and since the coefficient multiplying each Fourier coefficient from Lemma \ref{Fourinv} in each of $\Phi$ and $\Psi$ satisfies the growth condition from Definition \ref{Jacdef}, the same happens in the product. This proves the proposition.
\end{proof}
A typical case for Proposition \ref{prod} is the case where $L$ is the lattice $M(m)$ obtained from a lattice $M$ by multiplying the forms by the integer $m>0$, $K=M(n)$ for another integer $n>0$, $\iota=\operatorname{Id}_{M}$, and $w$ and $v$ come from the same element $u\in\operatorname{Gr}(M_{\mathbb{R}})$. Then $L+_{\iota}K=M(m+n)$ and $v+_{\iota}w=u$, and we obtain the expected addition of the indices. Note that in the positive definite case the Grassmannians are trivial, the isomorphism $\iota$ from Proposition \ref{prod} takes the form of choosing bases for $L$ and $K$ and expressing the bilinear forms in terms of positive definite integral-valued symmetric matrices (with even diagonal entries). Then $L+_{\iota}K$ is represented by the sum of these matrices, and Proposition \ref{prod}, in the holomorphic case, reproduces the product rule from Corollary 1.7 of \cite{[Zi]} and others (such a product rule also holds in the vector-valued setting, as appearing in Remark 3 of \cite{[W]}).

Combining Proposition \ref{prod} with Theorem \ref{main}, we obtain the following operation on pairs of vector-valued modular forms.
\begin{thm}
Take two lattices $L$ and $K$ of the same signature $(b_{+},b_{-})$, and fix a group isomorphism $\iota:L \to K$, with real extension $\iota_{\mathbb{R}}$. Assume that $\iota_{\mathbb{R}}$ takes the spaces $v_{\pm}$ associated with an element $v\in\operatorname{Gr}(L_{\mathbb{R}})$ onto the respective spaces $v_{\pm}$ corresponding to $w\in\operatorname{Gr}(K_{\mathbb{R}})$, and take $L+_{\iota}K$ and $v+_{\iota}w$ to be as in Proposition \ref{prod}. Then the image of the map taking $\lambda \in L$ to $(\lambda,\iota\lambda) \in L \oplus K$ presents $L+_{\iota}K$ as a primitive sub-lattice of $L \oplus K$, and $(v+_{\iota}w)_{\pm}$ is the intersection of the spaces $v_{\pm} \oplus w_{\pm}$ associated with $v \oplus w\in\operatorname{Gr}(L_{\mathbb{R}} \oplus K_{\mathbb{R}})$ with $(L+_{\iota}K)_{\mathbb{R}}$. Take a modular form $F$ of weight $\big(k-\frac{b_{+}}{2},l-\frac{b_{-}}{2}\big)$ and representation $\rho_{L}^{*}$ and a modular form $H$ of weight $\big(\kappa-\frac{b_{+}}{2},\mu-\frac{b_{-}}{2}\big)$ and representation $\rho_{K}^{*}$, both with respect to the same subgroup $\Gamma$ of $\operatorname{Mp}_{2}(\mathbb{Z})$, and let $\Phi_{L,v}^{F}$ and $\Phi_{K,w}^{H}$ be the corresponding Jacobi forms from Theorem \ref{main}. Then the modular form that is associated with the product $\Phi_{L,v}^{F}\Phi_{K,u}^{H}$ defined in Proposition \ref{prod} is the theta contraction $\Theta_{(L \oplus K,L+_{\iota}K)}\big(F \otimes H;(v+_{\iota}w)^{\perp}\big)$ from Equation \eqref{contdef}. \label{MFprodJac}
\end{thm}

\begin{proof}
The fact that $\{(\lambda,\iota\lambda)\,|\,\lambda \in L\}$ is isomorphic to $L+_{\iota}K$, its primitivity in $L \oplus K$, and the relation between $v \oplus w\in\operatorname{Gr}(L_{\mathbb{R}} \oplus K_{\mathbb{R}})$ and $v+_{\iota}w\in\operatorname{Gr}\big((L+_{\iota}K)_{\mathbb{R}}\big)$ are clear from the definitions in Proposition \ref{prod}. Now, Proposition \ref{dirsum} allows us to write \[(\Phi_{L,v}^{F}\Phi_{K,w}^{H})(\tau,\zeta)=\Phi_{L,v}^{F}(\tau,\zeta)\Phi_{K,w}^{H}(\tau,\iota\zeta)=\Phi_{L \oplus K,v \oplus w}^{F \otimes H}(\tau,\zeta+\iota\zeta)\] for $\tau\in\mathcal{H}$ and $\zeta \in L_{\mathbb{C}}$, and note that $(L+_{\iota}K)_{\mathbb{C}}$, as a subspace of $L_{\mathbb{C}} \oplus K_{\mathbb{C}}$, is just $\{\zeta+\iota\zeta\,|\,\zeta \in L_{\mathbb{C}}\}$, with orthogonal complement $\{\xi-\iota\xi\,|\,\xi \in L_{\mathbb{C}}\}$. Therefore the right hand side is $\operatorname{Res}^{L \oplus K}_{L+_{\iota}K}\Phi_{L \oplus K,v \oplus w}^{F \otimes H}(\tau,\xi)$ for $\zeta+\iota\zeta \in (L+_{\iota}K)_{\mathbb{C}}$ in the notation from Proposition \ref{Jaccont}, which also gives that the associated modular form is the asserted one. This proves the theorem.
\end{proof}

The explicit formulae from the theta contraction involved in Theorem \ref{MFprodJac} are not so trivial. We consider the most classical case, where $L$ and $K$ are positive definite of rank 1. Denoting by $A_{2}$ the lattice $\mathbb{Z}$ with a generator of quadratic value 1, we write $L=A_{2}(m)$ and $K=A_{2}(n)$ for positive integers $m$ and $n$. The discriminants $D_{L}$ and $D_{K}$ are $\mathbb{Z}/2m\mathbb{Z}$ and $\mathbb{Z}/2n\mathbb{Z}$ respectively, the Grassmannians are trivial, the corresponding theta functions and Jacobi theta functions from Equations \eqref{Thetadef} and \eqref{JacTheta} are just the holomorphic functions $\theta_{r+2m\mathbb{Z}}(\tau):=\sum_{h\in\mathbb{Z}+\frac{r}{2m}}\mathbf{e}(mh^{2}\tau)$ and $\theta_{r+2m\mathbb{Z}}(\tau,\zeta):=\sum_{h\in\mathbb{Z}+\frac{r}{2m}}\mathbf{e}(mh^{2}\tau+2m\zeta)$ (with $\zeta\in\mathbb{C}$) as components, and the same for $K$, $L+_{\iota}K=A_{2}(m+n)$, and any other such lattice. The latter lattice is the ``diagonal'' image inside the rank 2 lattice $L \oplus K$, and the components of the function $\Theta_{L \oplus K,L+_{\iota}K}$ from Theorem \ref{modThetaLM} appearing in the theta contraction from Equation \eqref{contdef} are based, via Equation \eqref{compTheta}, on the orthogonal complement of $L+_{\iota}K$ inside $L \oplus K$. But this rank 1 lattice is generated by an element of quadratic value $\frac{mn(m+n)}{d^{2}}$ for $d:=\gcd\{m,n\}$, and we get that for $r\in\mathbb{Z}/2m\mathbb{Z}$ and $s\in\mathbb{Z}/2n\mathbb{Z}$ there are $\frac{m+n}{d}$ classes $t\in\mathbb{Z}/2(m+n)\mathbb{Z}$ that are congruent to $r+s$ modulo $2d$. Moreover, for each such $t$ there exists some $l(r,s,t)\in\mathbb{Z}/\frac{mn(m+n)}{d^{2}}\mathbb{Z}$, which is congruent to $\frac{rn-sm}{d}$ modulo $\frac{2mn}{d}$, such that if $F=\sum_{r\in\mathbb{Z}/2m\mathbb{Z}}F_{r}\mathfrak{e}_{r}^{*}$ and $H=\sum_{s\in\mathbb{Z}/2n\mathbb{Z}}H_{s}\mathfrak{e}_{s}^{*}$ then the operation from Theorem \ref{MFprodJac} produces \[\sum_{r\in\mathbb{Z}/2m\mathbb{Z}}\sum_{s\in\mathbb{Z}/2n\mathbb{Z}}\sum_{\substack{t\in\mathbb{Z}/2(m+n)\mathbb{Z} \\ t \equiv r+s(\mathrm{mod\ }2d)}}F_{r}H_{s}\theta_{l(r,s,t)+\frac{2mn(m+n)}{d^{2}}\mathbb{Z}}\mathfrak{e}_{r}^{*}.\] Applying this to the constant functions, and going over to the associated Jacobi forms, we obtain, for each such $r$ and $s$, the equality \[\theta_{r+2m\mathbb{Z}}(\tau,\zeta)\theta_{s+2n\mathbb{Z}}(\tau,\zeta)=\sum_{\substack{t\in\mathbb{Z}/2(m+n)\mathbb{Z} \\ t \equiv r+s(\mathrm{mod\ }2d)}}\theta_{l(r,s,t)+2\frac{mn(m+n)}{d^{2}}\mathbb{Z}}(\tau)\theta_{t+2(m+n)\mathbb{Z}}(\tau,\zeta).\]

For an explicit example, take $m=n=d=1$, so that $m+n=\frac{mn(m+n)}{d^{2}}=2$, and write $\theta(\tau)$ and $\tilde{\theta}(\tau)$ for $\theta_{0+2\mathbb{Z}}(\tau)$ and $\theta_{1+2\mathbb{Z}}(\tau)$ respectively. Then with the index 2, the functions $\theta_{0+4\mathbb{Z}}(\tau)$ and $\theta_{2+4\mathbb{Z}}(\tau)$ are just $\theta(2\tau)$ and $\tilde{\theta}(2\tau)$ respectively, and the remaining functions $\theta_{1+4\mathbb{Z}}(\tau)$ and $\theta_{3+4\mathbb{Z}}(\tau)$ coincide to a single function, which we denote by $\hat{\theta}(\tau)$. The corresponding map from Theorem \ref{MFprodJac} is then defined by \[\mathfrak{e}_{0+2\mathbb{Z}}^{*}\otimes\mathfrak{e}_{0+2\mathbb{Z}}^{*}\mapsto\theta(2\tau)\mathfrak{e}_{0+4\mathbb{Z}}^{*}+\tilde{\theta}(2\tau)\mathfrak{e}_{2+4\mathbb{Z}}^{*},\] \[\mathfrak{e}_{1+2\mathbb{Z}}^{*}\otimes\mathfrak{e}_{1+2\mathbb{Z}}^{*}\mapsto\tilde{\theta}(2\tau)\mathfrak{e}_{0+4\mathbb{Z}}^{*}+\theta(2\tau)\mathfrak{e}_{2+4\mathbb{Z}}^{*},\] and \[\mathfrak{e}_{0+2\mathbb{Z}}^{*}\otimes\mathfrak{e}_{1+2\mathbb{Z}}^{*},\ \mathfrak{e}_{1+2\mathbb{Z}}^{*}\otimes\mathfrak{e}_{0+2\mathbb{Z}}^{*}\mapsto\hat{\theta}(\tau)(\mathfrak{e}_{1+4\mathbb{Z}}^{*}+\mathfrak{e}_{3+4\mathbb{Z}}^{*}).\] It would be interesting to relate the generalization of these formulae appearing in the previous paragraph to the theory of theta blocks from, e.g., \cite{[GSZ]}.

\smallskip

We conclude this section by remarking that formulae similar to those from Propositions \ref{dirsum} and \ref{Jaccont} and Theorem \ref{MFprodJac} hold for the Jacobi forms from Proposition \ref{skewhol}, with $F$ (resp. $H$) replaced by a modular form with the representation $\rho_{L}$ (resp. $\rho_{K}$). Combining this with Remark \ref{algscal} explains why in the isomorphisms involving skew-holomorphic Jacobi forms, in \cite{[Sk]} and others, the algebra structure involves multiplying by $h(-\overline{\tau})$ (which is the formula for $\tilde{h}(\tau)$ from Corollary \ref{mapMFs} in the case of the trivial representation).

\section{Jacobi Forms with Characteristics \label{Char}}

Recall that Section 4 of \cite{[Bor]} concerns theta functions that are more general than those from Equation \eqref{Thetadef}. We take, in addition to the parameters from that equation, two vectors $\alpha$ and $\beta$ from $L_{\mathbb{R}}$, and set \[\Theta_{L}\big(\tau;\textstyle{\binom{\alpha}{\beta}};v\big):=\displaystyle\sum_{\lambda \in L^{*}}\mathbf{e}\bigg(\tau\frac{(\lambda+\beta)_{v_{+}}^{2}}{2}+\overline{\tau}\frac{(\lambda+\beta)_{v_{-}}^{2}}{2}-\big(\lambda+\tfrac{\beta}{2},\alpha\big)\bigg)\mathfrak{e}_{\lambda+L}.\] Recall that the group $\operatorname{SL}_{2}(\mathbb{Z})\subseteq\operatorname{SL}_{2}(\mathbb{R})$, hence also $\operatorname{Mp}_{2}(\mathbb{Z})\subseteq\operatorname{Mp}_{2}(\mathbb{R})$, acts on such column vectors (more naturally, it acts on $\mathbb{R}^{2}$ and we view these vectors as elements of $\mathbb{R}^{2}\otimes_{\mathbb{R}}L_{\mathbb{R}}$). Then Theorem 4.1 of \cite{[Bor]} states that \[\Theta_{L}\big(A\tau;\textstyle{A\binom{\alpha}{\beta}};v\big)=\phi(\tau)^{b_{+}}\overline{\phi(\tau)}^{b_{-}}\rho_{L}(A,\phi)\Theta_{L}\big(\tau;\textstyle{\binom{\alpha}{\beta}};v\big)\] for every $(A,\phi)\in\operatorname{Mp}_{2}(\mathbb{Z})$. We combine this generalization with Equation \eqref{JacTheta} and set
\begin{equation}
\Theta_{L}\big(\tau,\zeta;\textstyle{\binom{\alpha}{\beta}};v\big)\!:=\!\displaystyle\sum_{\lambda \in L^{*}}\!\mathbf{e}\bigg(\tau\frac{(\lambda+\beta)_{v_{+}}^{2}}{2}+\overline{\tau}\frac{(\lambda+\beta)_{v_{-}}^{2}}{2}+(\lambda+\beta,\zeta)-\big(\lambda+\tfrac{\beta}{2},\alpha\big)\bigg)\mathfrak{e}_{\lambda+L}, \label{charTheta}
\end{equation}
for which we prove the following extension of Theorem 4.1 of \cite{[Bor]} and of our Propositions \ref{perTheta} and \ref{difeq} and Theorem \ref{modTheta}.
\begin{thm}
Take $\tau\in\mathcal{H}$, $v\in\operatorname{Gr}(L_{\mathbb{R}})$, $\zeta \in L_{\mathbb{C}}$, and $\alpha$ and $\beta$ from $L_{\mathbb{R}}$. Then if $\sigma$ and $\nu$ are in $L$ then $\Theta_{L}\big(\tau,\zeta+\tau\sigma_{v_{+}}+\overline{\tau}\sigma_{v_{-}}+\nu;\binom{\alpha}{\beta};v\big)$ equals \[\mathbf{e}\bigg(-\tau\frac{\sigma_{v_{+}}^{2}}{2}-\overline{\tau}\frac{\sigma_{v_{-}}^{2}}{2}-(\sigma,\zeta)+(\nu,\beta)+(\sigma,\alpha)\bigg)\Theta_{L}\big(\tau,\zeta;\textstyle{\binom{\alpha}{\beta}};v\big),\] and if $(A,\phi)\in\operatorname{Mp}_{2}(\mathbb{Z})$ then $\Theta_{L}\Big(A\tau,\tfrac{\zeta_{v_{+}}}{j(A,\tau)}+\tfrac{\zeta_{v_{-}}}{\overline{j(A,\tau)}};A\binom{\alpha}{\beta};v\Big)$ equals \[\phi(\tau)^{b_{+}}\overline{\phi(\tau)}^{b_{-}}\mathbf{e}\Big(\tfrac{j_{A}'\zeta_{v_{+}}^{2}}{2j(A,\tau)}+\tfrac{j_{A}'\zeta_{v_{-}}^{2}}{2\overline{j(A,\tau)}}\Big)\rho_{L}(A,\phi) \Theta_{L}\big(\tau,\zeta;\textstyle{\binom{\alpha}{\beta}};v\big).\] The differential operators $4\pi i\partial_{\tau}-\Delta_{v_{+}}^{h}$, $4\pi i\partial_{\overline{\tau}}-\Delta_{v_{-}}^{h}$, and $\partial_{\overline{\zeta}_{i}}$ for $1 \leq i \leq b_{+}+b_{-}$ from Proposition \ref{difeq} annihilate these theta functions as well. \label{propgen}
\end{thm}

\begin{proof}
For the first equality we follow the proof of Proposition \ref{perTheta}, with adding $\beta$ to each $\lambda$, where we have the additional term with $(\nu,\beta)$, and the index change from $\lambda$ to $\lambda+\sigma$ produces $(\sigma,\alpha)$ as well. Proving the second one for $T^{\pm1}$ and $S$ suffices, where for the former element the term associated with $\lambda$ is multiplied by $\mathbf{e}\big(\frac{\pm(\lambda+\beta)^{2}}{2}\big)$ because of adding $\pm1$ to $\tau$, but also by $\mathbf{e}\big(\mp\big(\lambda+\tfrac{\beta}{2},\beta\big)\big)$ because we have added $\pm\beta$ to $\alpha$. As this combines to $\mathbf{e}\big(\pm\frac{\lambda^{2}}{2}\big)$ once again, the equality for $T^{\pm}$ is proved, and as in the proof of Theorem \ref{modTheta}, we can restrict attention to the equality for $S$ with $\zeta \in L_{\mathbb{R}}$ by holomorphicity. Because the vector is now $\binom{-\beta}{\alpha}$, applying the argument leading to Equation \eqref{quadcomp} and writing $\lambda+\frac{\alpha}{2}$ as the sum of $\lambda+\alpha-\zeta$ and $\zeta-\frac{\alpha}{2}$ shows that the coefficient multiplying $\mathfrak{e}_{\gamma}$ on the left hand side is the desired exponent $\mathbf{e}\Big(\frac{\zeta_{v_{+}}^{2}}{2\tau}+\frac{\zeta_{v_{-}}^{2}}{2\overline{\tau}}\Big)$ times \[\sum_{\lambda \in L+\gamma}\mathbf{e}\bigg(-\frac{(\lambda+\alpha-\zeta)_{v_{+}}^{2}}{2\tau}-\frac{(\lambda+\alpha-\zeta)_{v_{-}}^{2}}{2\overline{\tau}}+(\lambda+\alpha-\zeta,\beta)+\big(\zeta-\tfrac{\alpha}{2},\beta\big)\bigg).\] An application of Equation \eqref{PSF}, with $\xi=\gamma-\zeta+\alpha$ and $\eta=\beta$, transforms the total expression into $\mathbf{e}\big(\frac{b_{-}-b_{+}}{8}\big)\tau^{b_{+}/2}\overline{\tau}^{b_{-}/2}\mathbf{e}\Big(\frac{\zeta_{v_{+}}^{2}}{2\tau}+\frac{\zeta_{v_{-}}^{2}}{2\overline{\tau}}\Big)\big/\sqrt{D_{L}}$ times \[\sum_{\mu \in L^{*}}\mathbf{e}\bigg(\tau\frac{(\mu+\beta)_{v_{+}}^{2}}{2}+\overline{\tau}\frac{(\mu+\beta)_{v_{-}}^{2}}{2}+(\mu,\zeta)-(\gamma,\mu)-(\alpha,\mu)+\big(\zeta-\tfrac{\alpha}{2},\beta\big)\bigg),\] and once again decomposing the sum over $L^{*}$ into the cosets of $L$ produces the required external coefficient and the terms coming from $\rho_{L}(S)$. As the summand associated with $\mu \in L+\delta$ is easily seen to be the one appearing in $\Theta_{L}\big(\tau,\zeta;\binom{\alpha}{\beta};v\big)$, this indeed gives the desired right hand side. The action of the differential operators is verified exactly as in the proof of Proposition \ref{difeq}. This proves the theorem.
\end{proof}

\smallskip

Following Lemma \ref{Fourinv} and Corollary \ref{decomtheta}, we now obtain the following result.
\begin{prop}
For every $\alpha$ and $\beta$, the map taking a real-analytic function $F:\mathcal{H}\to\mathbb{C}[D_{L(-1)}]$ to the function $(\tau,\zeta)\mapsto\big\langle\Theta_{L}\big(\tau,\zeta;\binom{\alpha}{\beta};v\big),F(\tau)\big\rangle_{L}$ is an isomorphism onto the space of smooth functions $\Phi$ on $\mathcal{H} \times L_{\mathbb{C}}$ that are holomorphic in $\zeta$ and satisfy \[\Phi(\tau,\zeta+\tau\sigma_{v_{+}}+\overline{\tau}\sigma_{v_{-}}+\nu)= \mathbf{e}\bigg(-\tau\frac{\sigma_{v_{+}}^{2}}{2}-\overline{\tau}\frac{\sigma_{v_{-}}^{2}}{2}-(\sigma,\zeta)+(\nu,\beta)+(\sigma,\alpha)\bigg)\Phi(\tau,\zeta)\] for every $\tau\in\mathcal{H}$, $\zeta \in L_{\mathbb{C}}$, and $\sigma$ and $\nu$ in $L$. \label{genper}
\end{prop}

\begin{proof}
The fact that $\Theta_{L}\big(\tau,\zeta;\binom{\alpha}{\beta};v\big)$ satisfies this functional equation, established in Theorem \ref{propgen}, implies that its pairing with any function $F$ satisfies it as well. Conversely, if $\Phi(\tau,\zeta)$ satisfies this equation then $\Phi(\tau,\zeta)\mathbf{e}\big(-(\beta,\zeta)\big)$ is invariant under translations from $L$. Thus, as in the proof of Lemma \ref{Fourinv} (but in which we now write $\tilde{f}_{\lambda}(\tau)$ as $\mathbf{e}\Big(\tau\frac{\lambda_{v_{+}}^{2}}{2}+\overline{\tau}\frac{\lambda_{v_{-}}^{2}}{2}-\big(\lambda+\frac{\beta}{2},\alpha\big)\Big)$ times $f_{\lambda}(\tau)$), we can write $\Phi(\tau,\zeta)$ as $\sum_{\lambda \in L^{*}}f_{\lambda}(\tau)\mathbf{e}\Big(\tau\frac{(\lambda+\beta)_{v_{+}}^{2}}{2}+\overline{\tau}\frac{(\lambda+\beta)_{v_{-}}^{2}}{2}+(\lambda+\beta,\zeta)-\big(\lambda+\frac{\beta}{2},\alpha\big)\Big)$. Still following the proof of Lemma \ref{Fourinv}, we fix $\sigma \in L$ and substitute this expansion of $\Phi$ into both sides for the functional equation and applying a translation of $\sigma$ between the two summation indices shows that $f_{\lambda}$ depends only on $\lambda+L$. The desired result now follows as in the proof of Corollary \ref{decomtheta}. This prove the proposition.
\end{proof}

\begin{rmk}
As in Proposition \ref{pshol}, the isomorphism from Proposition \ref{genper} takes the holomorphic functions $F:\mathcal{H}\to\mathbb{C}[D_{L(-1)}]$ onto those functions $\Phi$ that are also annihilated by the operator $4\pi i\partial_{\overline{\tau}}-\Delta_{v_{-}}^{h}$ from Proposition \ref{difeq} and Theorem \ref{propgen}. Moreover, combining the isomorphism from Proposition \ref{genper} with the map $G\mapsto\tilde{G}:\tau\mapsto\omega_{L}\big(G(-\overline{\tau})\big)$ from Corollary \ref{mapMFs} gives an isomorphism between the functions on $\mathcal{H} \times L_{\mathbb{C}}$ satisfying the condition from that proposition and real-analytic functions $G:\mathcal{H}\to\mathbb{C}[D_{L}]$. Under this isomorphism, the function $G$ is holomorphic on $\mathcal{H}$ if and only if $\Phi$ is annihilated by $4\pi i\partial_{\tau}-\Delta_{v_{+}}^{h}$, as in Proposition \ref{skewhol}. \label{skewchar}
\end{rmk}
Note, however, that modularity properties of $F$ and the pairing from Proposition \ref{genper} are no longer related, because of the action of $\operatorname{Mp}_{2}(\mathbb{Z})$ on the column vectors $\binom{\alpha}{\beta}$. The same applies to $G$ as appearing in Remark \ref{skewchar}.

\smallskip

The theta functions from Equation \eqref{charTheta} closely resemble theta functions with characteristics, as defined in, e.g., Definition 1.7 of \cite{[FZ]}, among earlier references. For this we shall define $\hat{\Theta}_{L}\big(\tau,\zeta;\binom{\alpha}{\beta};v\big)$ to be $\mathbf{e}\big(-\frac{(\alpha,\beta)}{2}\big)\Theta_{L}\big(\tau,\zeta;\binom{\alpha}{\beta};v\big)$, as with this normalization the following properties are better-behaved. To put it in a broader context, if $L$ is positive definite of rank $g$, with a Gram matrix $\mathcal{M}$ using some basis of $L$ over $\mathbb{Z}$, then the Grassmannian is trivial, and $\tau\mapsto\tau\mathcal{M}$ is an embedding of $\mathcal{H}$ into the Siegel upper half-plane of degree $g$. Then we can consider $\alpha$ and $\beta$ as the corresponding elements of $\mathbb{R}^{g}$, view $D_{L}$ as a subgroup of $L_{\mathbb{R}}/L\cong\mathbb{R}^{g}/\mathbb{Z}^{g}$ via the basis, and by lifting each $\gamma \in D_{L}$ to an appropriate element of $\mathbb{R}^{g}$, the coefficient $\hat{\theta}_{L+\gamma}\big(\tau,\zeta;\binom{\alpha}{\beta}\big)$ is the theta function $\theta\big[{2\beta+2\gamma \atop -2\alpha}\big](\mathcal{M}\zeta,\tau\mathcal{M})$ in the notation of \cite{[FZ]}. Hence our functions $\hat{\theta}_{L+\gamma}\big(\tau,\zeta;\binom{\alpha}{\beta},v\big)$ are a generalization of this restriction of theta functions with characteristics to the indefinite case.

Indeed, the usual arguments show that the theta function $\hat{\Theta}_{L}\big(\tau,\zeta;\binom{\alpha}{\beta};v\big)$ satisfies the following properties of theta functions with characteristics: The parity formula $\hat{\Theta}_{L}\big(\tau,\zeta;\binom{\alpha}{\beta};v\big)=\hat{\Theta}_{L}\big(\tau,-\zeta;\binom{-\alpha}{-\beta};v\big)$, the periodicity relation
\begin{equation}
\hat{\Theta}_{L}\big(\tau,\zeta;\textstyle{\binom{\alpha+\nu}{\beta+\sigma}};v\big)=\displaystyle\mathbf{e}\big(-(\beta,\nu)\big)\hat{\Theta}_{L}\big(\tau,\zeta;\textstyle{\binom{\alpha}{\beta}};v\big)\qquad\mathrm{for\ }\nu\mathrm{\ and\ }\sigma\mathrm{\ in\ }L \label{perofchar}
\end{equation}
(by a summation index change), and the general relation \[\hat{\Theta}_{L}(\tau,\!\zeta;\!\textstyle{\binom{\alpha+\nu}{\beta+\sigma}};\!v)\!=\!\displaystyle\mathbf{e}\bigg(\!\tau\frac{\sigma_{v_{+}}^{2}}{2}+\overline{\tau}\frac{\sigma_{v_{-}}^{2}}{2}+ (\sigma,\zeta-\alpha-\nu)\!\bigg)\hat{\Theta}_{L}(\tau,\!\zeta+\tau\sigma_{v_{+}}+\overline{\tau}\sigma_{v_{-}}-\nu;\!\textstyle{\binom{\alpha}{\beta}};v)\] with respect to changing the variable, in which $\sigma$ and $\nu$ are arbitrary in $L_{\mathbb{R}}$ (the proof is similar to that of the first equality in Theorem \ref{propgen}). In particular it follows from Equation \eqref{perofchar} that if $\beta \in L^{*}$ then the theta function is defined with characteristics coming from cosets modulo $L$. Then, for $\alpha$ and $\beta$ in $D_{L}$ the function $\hat{\theta}_{L+\gamma}\big(\tau,\zeta;\binom{\alpha}{\beta};v\big)$ is just $\mathbf{e}\big(-(\alpha,\beta+\gamma)\big)$ times the initial component $\theta_{L+\beta+\gamma}(\tau,\zeta;v)$ of Equation \eqref{JacTheta}. This shows, in fact, that in this setting the vector-valued nature of $\hat{\Theta}_{L}$ contains no more information than any single scalar-valued component.

\smallskip

While the lack of modularity of the functions from Proposition \ref{genper} makes their applicability less evident, we can present the following application, which complements Proposition \ref{uparrow}.
\begin{prop}
Assume that $L$ is an over-lattice of another lattice $\Lambda$, and set $H=L/\Lambda \subseteq D_{\Lambda}$. Take $\tau\in\mathcal{H}$, $v$ in the common Grassmannian of $L_{\mathbb{R}}$ and $\Lambda_{\mathbb{R}}$, and $\zeta \in L_{\mathbb{C}}=\Lambda_{\mathbb{C}}$. Then $\uparrow^{L}_{\Lambda}\Theta_{L}(\tau,\zeta;v)$ equals $\frac{1}{|H|}\sum_{\alpha \in H}\sum_{\beta \in H}\hat{\Theta}_{\Lambda}\big(\tau,\zeta;\binom{\alpha}{\beta};v\big)$. Moreover, if $F$ is a modular form of weight $(k,l)$ and representation $\rho_{\Lambda}^{*}$ with respect to $\Gamma\subseteq\operatorname{Mp}_{2}(\mathbb{Z})$, then we have the equality \[\Phi_{L,v}^{\downarrow^{L(-1)}_{\Lambda(-1)}F}(\tau,\zeta)=\frac{1}{|H|}\sum_{\alpha \in H}\sum_{\beta \in H}\big\langle\hat{\Theta}_{\Lambda}\big(\tau,\zeta;\textstyle{\binom{\alpha}{\beta}};v\big),F(\tau)\big\rangle_{\Lambda}.\] \label{downarrow}
\end{prop}

\begin{proof}
The definition of $\uparrow$ and the fact that $\Theta_{L}(\tau,\zeta;v)=\downarrow^{L}_{\Lambda}\Theta_{L}(\tau,\zeta;v)$ implies that for $\gamma \in D_{\Lambda}$, the coefficient of $\mathfrak{e}_{\gamma}$ in $\uparrow^{L}_{\Lambda}\Theta_{L}(\tau,\zeta;v)$ is $\sum_{\beta \in H}\theta_{\Lambda+\gamma+\beta}(\tau,\zeta;v)$ when $\gamma \in H^{\perp}$ and 0 otherwise. Moreover, as $\sum_{\alpha \in H}\mathbf{e}\big(-(\alpha,\gamma)\big)$ equals $|H|$ when $\gamma \in H^{\perp}$ and 0 otherwise, and $H$ is isotropic, we can express our coefficient as $\frac{1}{|H|}\sum_{\alpha \in H}\sum_{\beta \in H}\mathbf{e}\big(-(\alpha,\beta+\gamma)\big)\theta_{\Lambda+\gamma+\beta}(\tau,\zeta;v)$. The fact that the summand associated with $\alpha$ and $\beta$ was seen above to be $\hat{\theta}_{L+\gamma}\big(\tau,\zeta;\binom{\alpha}{\beta};v\big)$ thus establishes the first assertion. For the second one, we note that the definition of the left hand side in Theorem \ref{main} combines with Lemma 2.2 of \cite{[Ze3]} to express the left hand side as $\Phi_{L,v}^{F}(\tau,\zeta):=\big\langle\uparrow^{L}_{\Lambda}\Theta_{L}(\tau,\zeta;v),F(\tau)\big\rangle_{L}$, and substituting the expression from the first assertion produces the desired right hand side. This proves the proposition.
\end{proof}
As for Propositions \ref{dirsum} and \ref{Jaccont} and Theorem \ref{MFprodJac}, combining Proposition \ref{downarrow} with Proposition \ref{skewhol} and the commutation relation $\omega_{L}\circ\downarrow^{L}_{\Lambda}=\downarrow^{L(-1)}_{\Lambda(-1)}\circ\omega_{\Lambda}$ also yields, for a modular form $G$ of weight $(k,l)$ and representation $\rho_{\Lambda}$ with respect to the group $\tilde{\Gamma}$ from Corollary \ref{mapMFs}, the equality \[\tilde{\Phi}_{L,v}^{\downarrow^{L}_{\Lambda}G}(\tau,\zeta)=\frac{1}{|H|}\sum_{\alpha \in H}\sum_{\beta \in H}\big\langle\hat{\Theta}_{\Lambda}\big(\tau,\zeta;\textstyle{\binom{\alpha}{\beta}};v\big),\omega_{\Lambda}\big(G(-\overline{\tau})\big)\big\rangle_{\Lambda}.\]

\noindent\textsc{Einstein Institute of Mathematics, the Hebrew University of Jerusalem, Edmund Safra Campus, Jerusalem 91904, Israel}

\noindent E-mail address: zemels@math.huji.ac.il


\begin{thebibliography}{}{}

\bibitem[A]{[A]} Ajouz, A., \textsc{Hecke Operators on Jacobi Forms of Lattice Index and the Relation to Elliptic Modular Forms}, Ph.D. thesis, University of Siegen (2015).
\bibitem[BK]{[BK]} B\"{o}cherer, S., Kohnen, W., \textsc{Estimates for Fourier Coefficients of Siegel Cusp Forms}. Math. Ann., vol 297 issue 1, 499–-517 (1993).
\bibitem[Bor]{[Bor]} Borcherds, R. E., \textsc{Automorphic Forms with Singularities on Grassmannians}, Invent. Math., vol. 132, 491--562 (1998).
\bibitem[Boy]{[Boy]} Boylan, H., \textsc{Jacobi Forms, Finite Quadratic Modules and Weil Representations over Number Fields}, Lecture Notes in Mathematics 2130, Springer International Publishing, Switzerland, xix+130pp (2015).
\bibitem[BRR]{[BRR]} Bringmann, K., Raum, M., Richter, O. \textsc{Harmonic Maass-Jacobi Forms with Singularities and a Theta-Like Decomposition}, Trans. Math. Amer. Soc., vol 367, 6647--6670 (2015).
\bibitem[BR]{[BR]} Bringmann, K., Richter, O., \textsc{Zagier-Type Dualities and Lifting Maps for Harmonic Maass-Jacobi Forms}, Adv. Math., vol. 225 no. 4, 2298--2315 (2010).
\bibitem[CWR]{[CWR]} Conrey, C., Westerholt-Raum, M., \textsc{Harmonic Maaß-Jacobi Forms of Degree 1 with Higher Rank Indices}, Int. J. Number Theory, vol 12 no. 7, 1871--1897 (2016).
\bibitem[EZ]{[EZ]} Eichler, M., Zagier, D., \textsc{The Theory of Jacobi Forms}, Progress in Mathematics 55, Birkh\"{a}user, Boston, Basel, Stuttgart (1985).
\bibitem[FZ]{[FZ]} Farkas, H. M., Zemel, S., \textsc{Generalizations of Thomae's Formula for $Z_{n}$ curves}, DEVM 21, Springer--Verlag, xi+354pp (2011).
\bibitem[GSZ]{[GSZ]} Gritsenko, V., Skoruppa, N. P., Zagier, D., \textsc{Theta Blocks}, pre-print, https://arxiv.org/abs/1907.00188 (2019).
\bibitem[H]{[H]} Hayashida, S., \textsc{Skew-Holomorphic Jacobi Forms of Higher Degree}, in: Automorphic Forms and Zeta Functions--in Memory of Tsuneo Arakawa, 130-139, World Scientific (2006).
\bibitem[K1]{[K1]} Kohnen, W., \textsc{Modular Forms of Half-Integral Weight on $\Gamma_{0}(4)$}, Math. Ann., vol 248, 249--266 (1980).
\bibitem[K2]{[K2]} Kohnen, W., \textsc{Newforms of Half-Integral Weight}, J. Reine Angew. Math., vol 333, 32--72 (1982).
\bibitem[LZ]{[LZ]} Li, Y., Zemel, S., \textsc{Shimura Lift of Weakly Holomorphic Modular Forms}, Math. Z., vol 290, 37--61 (2018).
\bibitem[Ma]{[Ma]} Ma, S., \textsc{Quasi--Pullback of Borcherds Products}, Bull. London Math. Soc., vol. 51 issue 6, 1061--1078 (2019).
\bibitem[Mo]{[Mo]} Mocanu, A., \textsc{On Jacobi Forms of Lattice Index}, Ph.D. thesis, University of Nottingham (2019).
\bibitem[RR]{[RR]} Raum, M., Richter, O. \textsc{The Skew-Maass Lift I}, Res. Math. Sci., vol 6 paper 22, 1--59 (2019).
\bibitem[Sch]{[Sch]} Scheithauer, N. R., \textsc{The Weil representation of $SL_{2}(\mathbb{Z})$ and some applications}, Int. Math. Res. Not., no. 8, 1488--1545 (2009).
\bibitem[Sk]{[Sk]} Skoruppa, N. P., \textsc{Developments in the Theory of Jacobi Forms}, in: Automorphic Functions and their Applications, 167–185, Acad. Sci. USSR, Inst. Appl. Math., Khabarovsk (1990).
\bibitem[Str]{[Str]} Str\"{o}mberg, F., \textsc{Weil Representations Associated to Finite Quadratic Modules}, Math. Z., vol 275 issue 1, 509--527 (2013).
\bibitem[WR]{[WR]} Westerholt-Raum, M., \textsc{H-Harmonic Maaß--Jacobi Forms of Degree 1: The Analytic Theory of Some Indefinite Theta Series}, Res. Math. Sci., vol 2 paper 12, 1--34 (2015).
\bibitem[W]{[W]} Williams, B., \textsc{Remarks on the Theta Decomposition of Vector-Valued Jacobi Forms}, J. Number Theory, vol 197, 250--267 (2019).
\bibitem[Ze1]{[Ze1]} Zemel, S., \textsc{A $p$-adic Approach to the Weil Representation of Discriminant Forms Arising from Even Lattices}, Math. Ann. Qu\'{e}bec, vol 39 issue 1, 61--89 (2015).
\bibitem[Ze2]{[Ze2]} Zemel, S., \textsc{Weight Changing Operators for Automorphic Forms on Grassmannians and Differential Properties of Certain Theta Lifts}, Nagoya Math. J., vol 228, 186--221 (2017).
\bibitem[Ze3]{[Ze3]} Zemel, S., \textsc{Seesaw Identities and Theta Contractions with Generalized Theta Functions, and Restrictions of Theta Lifts}, pre-print, https://arxiv.org/abs/2009.06012 (2020).
\bibitem[Zi]{[Zi]} Ziegler, C., \textsc{Jacobi Forms of Higher Degree}, Abh. Math. Semin. Univ. Hambg., vol 59, 191--224 (1989).

\end{thebibliography}
\end{document}